\documentclass[11pt]{article}
\usepackage{xcolor}
\usepackage{amscd}
\usepackage{graphicx,epsfig}
\usepackage{xcolor}
\usepackage{comment}
\usepackage{tikz}
\usepackage{hyperref}
\usepackage{amsmath,amscd,amsthm,amsfonts,amssymb}
\usepackage{xcolor}
\usepackage{cases}

	 \newcommand{\rn}{\mathbb{R}^n}

	  \newcommand{\et}{\mathbb{E}^3}

	\newtheorem{thm}{Theorem}
	\newtheorem*{thm*}{Theorem}
	\newtheorem{cor}{Corollary}
	\newtheorem{lem}{Lemma}
	\newtheorem{rem}{Remark}

	\newtheorem{prop}{Proposition}
	\newtheorem{defn}{Definition}
	\newtheorem{notation}{Terminology}

	\newtheorem{ex}{Example}

\def\XXint#1#2#3{{\setbox0=\hbox{$#1{#2#3}{\int}$ }
\vcenter{\hbox{$#2#3$ }}\kern-.6\wd0}}

    \usepackage[breakable]{tcolorbox}
    \usepackage{parskip} 
    
    \usepackage{iftex}
    \ifPDFTeX
    	\usepackage[T1]{fontenc}
    	\usepackage{mathpazo}
    \else
    	\usepackage{fontspec}
    \fi

    \usepackage{graphicx}
    
    \usepackage{caption}
    \DeclareCaptionFormat{nocaption}{}
    \usepackage{float}
    \usepackage{xcolor} 
    \usepackage{enumerate} 
    \usepackage{geometry} 
    \usepackage{amsmath} 
    \usepackage{amssymb} 
    \usepackage{textcomp} 
    \AtBeginDocument{%
    }
    \usepackage{upquote} 
    \usepackage{eurosym} 
    \usepackage[mathletters]{ucs} 
    \usepackage{fancyvrb} 
    \usepackage{grffile} 
    \makeatletter 
    \@ifpackagelater{grffile}{2019/11/01}
    {
    }
    {
      \def\Gread@@xetex#1{%
        \IfFileExists{"\Gin@base".bb}%
        {\Gread@eps{\Gin@base.bb}}%
        {\Gread@@xetex@aux#1}%
      }
    }
    \makeatother
    \usepackage[Export]{adjustbox} 
    \adjustboxset{max size={0.9\linewidth}{0.9\paperheight}}

    \usepackage{hyperref}
    \usepackage{titling}
    \usepackage{longtable} 
    \usepackage{booktabs}  
    \usepackage[inline]{enumitem} 
    \usepackage[normalem]{ulem} 
    \usepackage{mathrsfs}

    \definecolor{urlcolor}{rgb}{0,.145,.698}
    \definecolor{linkcolor}{rgb}{.71,0.21,0.01}
    \definecolor{citecolor}{rgb}{.12,.54,.11}

    \definecolor{ansi-black}{HTML}{3E424D}
    \definecolor{ansi-black-intense}{HTML}{282C36}
    \definecolor{ansi-red}{HTML}{E75C58}
    \definecolor{ansi-red-intense}{HTML}{B22B31}
    \definecolor{ansi-green}{HTML}{00A250}
    \definecolor{ansi-green-intense}{HTML}{007427}
    \definecolor{ansi-yellow}{HTML}{DDB62B}
    \definecolor{ansi-yellow-intense}{HTML}{B27D12}
    \definecolor{ansi-blue}{HTML}{208FFB}
    \definecolor{ansi-blue-intense}{HTML}{0065CA}
    \definecolor{ansi-magenta}{HTML}{D160C4}
    \definecolor{ansi-magenta-intense}{HTML}{A03196}
    \definecolor{ansi-cyan}{HTML}{60C6C8}
    \definecolor{ansi-cyan-intense}{HTML}{258F8F}
    \definecolor{ansi-white}{HTML}{C5C1B4}
    \definecolor{ansi-white-intense}{HTML}{A1A6B2}
    \definecolor{ansi-default-inverse-fg}{HTML}{FFFFFF}
    \definecolor{ansi-default-inverse-bg}{HTML}{000000}

    \definecolor{outerrorbackground}{HTML}{FFDFDF}

    \DefineVerbatimEnvironment{Highlighting}{Verbatim}{commandchars=\\\{\}}

    \let\Oldtex\TeX
    \let\Oldlatex\LaTeX
    \renewcommand{\TeX}{\textrm{\Oldtex}}
    \renewcommand{\LaTeX}{\textrm{\Oldlatex}}
    \title{Untitled}

\makeatletter
\def\PY@reset{\let\PY@it=\relax \let\PY@bf=\relax%
    \let\PY@ul=\relax \let\PY@tc=\relax%
    \let\PY@bc=\relax \let\PY@ff=\relax}
\def\PY@tok#1{\csname PY@tok@#1\endcsname}
\def\PY@toks#1+{\ifx\relax#1\empty\else%
    \PY@tok{#1}\expandafter\PY@toks\fi}
\def\PY@do#1{\PY@bc{\PY@tc{\PY@ul{%
    \PY@it{\PY@bf{\PY@ff{#1}}}}}}}
\def\PY#1#2{\PY@reset\PY@toks#1+\relax+\PY@do{#2}}

\@namedef{PY@tok@w}{\def\PY@tc##1{\textcolor[rgb]{0.73,0.73,0.73}{##1}}}
\@namedef{PY@tok@c}{\let\PY@it=\textit\def\PY@tc##1{\textcolor[rgb]{0.25,0.50,0.50}{##1}}}
\@namedef{PY@tok@cp}{\def\PY@tc##1{\textcolor[rgb]{0.74,0.48,0.00}{##1}}}
\@namedef{PY@tok@k}{\let\PY@bf=\textbf\def\PY@tc##1{\textcolor[rgb]{0.00,0.50,0.00}{##1}}}
\@namedef{PY@tok@kp}{\def\PY@tc##1{\textcolor[rgb]{0.00,0.50,0.00}{##1}}}
\@namedef{PY@tok@kt}{\def\PY@tc##1{\textcolor[rgb]{0.69,0.00,0.25}{##1}}}
\@namedef{PY@tok@o}{\def\PY@tc##1{\textcolor[rgb]{0.40,0.40,0.40}{##1}}}
\@namedef{PY@tok@ow}{\let\PY@bf=\textbf\def\PY@tc##1{\textcolor[rgb]{0.67,0.13,1.00}{##1}}}
\@namedef{PY@tok@nb}{\def\PY@tc##1{\textcolor[rgb]{0.00,0.50,0.00}{##1}}}
\@namedef{PY@tok@nf}{\def\PY@tc##1{\textcolor[rgb]{0.00,0.00,1.00}{##1}}}
\@namedef{PY@tok@nc}{\let\PY@bf=\textbf\def\PY@tc##1{\textcolor[rgb]{0.00,0.00,1.00}{##1}}}
\@namedef{PY@tok@nn}{\let\PY@bf=\textbf\def\PY@tc##1{\textcolor[rgb]{0.00,0.00,1.00}{##1}}}
\@namedef{PY@tok@ne}{\let\PY@bf=\textbf\def\PY@tc##1{\textcolor[rgb]{0.82,0.25,0.23}{##1}}}
\@namedef{PY@tok@nv}{\def\PY@tc##1{\textcolor[rgb]{0.10,0.09,0.49}{##1}}}
\@namedef{PY@tok@no}{\def\PY@tc##1{\textcolor[rgb]{0.53,0.00,0.00}{##1}}}
\@namedef{PY@tok@nl}{\def\PY@tc##1{\textcolor[rgb]{0.63,0.63,0.00}{##1}}}
\@namedef{PY@tok@ni}{\let\PY@bf=\textbf\def\PY@tc##1{\textcolor[rgb]{0.60,0.60,0.60}{##1}}}
\@namedef{PY@tok@na}{\def\PY@tc##1{\textcolor[rgb]{0.49,0.56,0.16}{##1}}}
\@namedef{PY@tok@nt}{\let\PY@bf=\textbf\def\PY@tc##1{\textcolor[rgb]{0.00,0.50,0.00}{##1}}}
\@namedef{PY@tok@nd}{\def\PY@tc##1{\textcolor[rgb]{0.67,0.13,1.00}{##1}}}
\@namedef{PY@tok@s}{\def\PY@tc##1{\textcolor[rgb]{0.73,0.13,0.13}{##1}}}
\@namedef{PY@tok@sd}{\let\PY@it=\textit\def\PY@tc##1{\textcolor[rgb]{0.73,0.13,0.13}{##1}}}
\@namedef{PY@tok@si}{\let\PY@bf=\textbf\def\PY@tc##1{\textcolor[rgb]{0.73,0.40,0.53}{##1}}}
\@namedef{PY@tok@se}{\let\PY@bf=\textbf\def\PY@tc##1{\textcolor[rgb]{0.73,0.40,0.13}{##1}}}
\@namedef{PY@tok@sr}{\def\PY@tc##1{\textcolor[rgb]{0.73,0.40,0.53}{##1}}}
\@namedef{PY@tok@ss}{\def\PY@tc##1{\textcolor[rgb]{0.10,0.09,0.49}{##1}}}
\@namedef{PY@tok@sx}{\def\PY@tc##1{\textcolor[rgb]{0.00,0.50,0.00}{##1}}}
\@namedef{PY@tok@m}{\def\PY@tc##1{\textcolor[rgb]{0.40,0.40,0.40}{##1}}}
\@namedef{PY@tok@gh}{\let\PY@bf=\textbf\def\PY@tc##1{\textcolor[rgb]{0.00,0.00,0.50}{##1}}}
\@namedef{PY@tok@gu}{\let\PY@bf=\textbf\def\PY@tc##1{\textcolor[rgb]{0.50,0.00,0.50}{##1}}}
\@namedef{PY@tok@gd}{\def\PY@tc##1{\textcolor[rgb]{0.63,0.00,0.00}{##1}}}
\@namedef{PY@tok@gi}{\def\PY@tc##1{\textcolor[rgb]{0.00,0.63,0.00}{##1}}}
\@namedef{PY@tok@gr}{\def\PY@tc##1{\textcolor[rgb]{1.00,0.00,0.00}{##1}}}
\@namedef{PY@tok@ge}{\let\PY@it=\textit}
\@namedef{PY@tok@gs}{\let\PY@bf=\textbf}
\@namedef{PY@tok@gp}{\let\PY@bf=\textbf\def\PY@tc##1{\textcolor[rgb]{0.00,0.00,0.50}{##1}}}
\@namedef{PY@tok@go}{\def\PY@tc##1{\textcolor[rgb]{0.53,0.53,0.53}{##1}}}
\@namedef{PY@tok@gt}{\def\PY@tc##1{\textcolor[rgb]{0.00,0.27,0.87}{##1}}}
\@namedef{PY@tok@err}{\def\PY@bc##1{{\setlength{\fboxsep}{\string -\fboxrule}\fcolorbox[rgb]{1.00,0.00,0.00}{1,1,1}{\strut ##1}}}}
\@namedef{PY@tok@kc}{\let\PY@bf=\textbf\def\PY@tc##1{\textcolor[rgb]{0.00,0.50,0.00}{##1}}}
\@namedef{PY@tok@kd}{\let\PY@bf=\textbf\def\PY@tc##1{\textcolor[rgb]{0.00,0.50,0.00}{##1}}}
\@namedef{PY@tok@kn}{\let\PY@bf=\textbf\def\PY@tc##1{\textcolor[rgb]{0.00,0.50,0.00}{##1}}}
\@namedef{PY@tok@kr}{\let\PY@bf=\textbf\def\PY@tc##1{\textcolor[rgb]{0.00,0.50,0.00}{##1}}}
\@namedef{PY@tok@bp}{\def\PY@tc##1{\textcolor[rgb]{0.00,0.50,0.00}{##1}}}
\@namedef{PY@tok@fm}{\def\PY@tc##1{\textcolor[rgb]{0.00,0.00,1.00}{##1}}}
\@namedef{PY@tok@vc}{\def\PY@tc##1{\textcolor[rgb]{0.10,0.09,0.49}{##1}}}
\@namedef{PY@tok@vg}{\def\PY@tc##1{\textcolor[rgb]{0.10,0.09,0.49}{##1}}}
\@namedef{PY@tok@vi}{\def\PY@tc##1{\textcolor[rgb]{0.10,0.09,0.49}{##1}}}
\@namedef{PY@tok@vm}{\def\PY@tc##1{\textcolor[rgb]{0.10,0.09,0.49}{##1}}}
\@namedef{PY@tok@sa}{\def\PY@tc##1{\textcolor[rgb]{0.73,0.13,0.13}{##1}}}
\@namedef{PY@tok@sb}{\def\PY@tc##1{\textcolor[rgb]{0.73,0.13,0.13}{##1}}}
\@namedef{PY@tok@sc}{\def\PY@tc##1{\textcolor[rgb]{0.73,0.13,0.13}{##1}}}
\@namedef{PY@tok@dl}{\def\PY@tc##1{\textcolor[rgb]{0.73,0.13,0.13}{##1}}}
\@namedef{PY@tok@s2}{\def\PY@tc##1{\textcolor[rgb]{0.73,0.13,0.13}{##1}}}
\@namedef{PY@tok@sh}{\def\PY@tc##1{\textcolor[rgb]{0.73,0.13,0.13}{##1}}}
\@namedef{PY@tok@s1}{\def\PY@tc##1{\textcolor[rgb]{0.73,0.13,0.13}{##1}}}
\@namedef{PY@tok@mb}{\def\PY@tc##1{\textcolor[rgb]{0.40,0.40,0.40}{##1}}}
\@namedef{PY@tok@mf}{\def\PY@tc##1{\textcolor[rgb]{0.40,0.40,0.40}{##1}}}
\@namedef{PY@tok@mh}{\def\PY@tc##1{\textcolor[rgb]{0.40,0.40,0.40}{##1}}}
\@namedef{PY@tok@mi}{\def\PY@tc##1{\textcolor[rgb]{0.40,0.40,0.40}{##1}}}
\@namedef{PY@tok@il}{\def\PY@tc##1{\textcolor[rgb]{0.40,0.40,0.40}{##1}}}
\@namedef{PY@tok@mo}{\def\PY@tc##1{\textcolor[rgb]{0.40,0.40,0.40}{##1}}}
\@namedef{PY@tok@ch}{\let\PY@it=\textit\def\PY@tc##1{\textcolor[rgb]{0.25,0.50,0.50}{##1}}}
\@namedef{PY@tok@cm}{\let\PY@it=\textit\def\PY@tc##1{\textcolor[rgb]{0.25,0.50,0.50}{##1}}}
\@namedef{PY@tok@cpf}{\let\PY@it=\textit\def\PY@tc##1{\textcolor[rgb]{0.25,0.50,0.50}{##1}}}
\@namedef{PY@tok@c1}{\let\PY@it=\textit\def\PY@tc##1{\textcolor[rgb]{0.25,0.50,0.50}{##1}}}
\@namedef{PY@tok@cs}{\let\PY@it=\textit\def\PY@tc##1{\textcolor[rgb]{0.25,0.50,0.50}{##1}}}

\makeatother

    \makeatletter
        \newbox\Wrappedcontinuationbox 
        \newbox\Wrappedvisiblespacebox 
        \newcommand*\Wrappedvisiblespace {\textcolor{red}{\textvisiblespace}} 
        \newcommand*\Wrappedcontinuationsymbol {\textcolor{red}{\llap{\tiny$\m@th\hookrightarrow$}}} 
        \newcommand*\Wrappedcontinuationindent {3ex } 
        \newcommand*\Wrappedafterbreak {\kern\Wrappedcontinuationindent\copy\Wrappedcontinuationbox} 
        \newcommand*\Wrappedbreaksatspecials {%
            \def\PYGZus{\discretionary{\char`\_}{\Wrappedafterbreak}{\char`\_}}%
            \def\PYGZob{\discretionary{}{\Wrappedafterbreak\char`\{}{\char`\{}}%
            \def\PYGZcb{\discretionary{\char`\}}{\Wrappedafterbreak}{\char`\}}}%
            \def\PYGZca{\discretionary{\char`\^}{\Wrappedafterbreak}{\char`\^}}%
            \def\PYGZam{\discretionary{\char`\&}{\Wrappedafterbreak}{\char`\&}}%
            \def\PYGZlt{\discretionary{}{\Wrappedafterbreak\char`\<}{\char`\<}}%
            \def\PYGZgt{\discretionary{\char`\>}{\Wrappedafterbreak}{\char`\>}}%
            \def\PYGZsh{\discretionary{}{\Wrappedafterbreak\char`\#}{\char`\#}}%
            \def\PYGZpc{\discretionary{}{\Wrappedafterbreak\char`\%}{\char`\%}}%
            \def\PYGZdl{\discretionary{}{\Wrappedafterbreak\char`\$}{\char`\$}}%
            \def\PYGZhy{\discretionary{\char`\-}{\Wrappedafterbreak}{\char`\-}}%
            \def\PYGZsq{\discretionary{}{\Wrappedafterbreak\textquotesingle}{\textquotesingle}}%
            \def\PYGZdq{\discretionary{}{\Wrappedafterbreak\char`\"}{\char`\"}}%
            \def\PYGZti{\discretionary{\char`\~}{\Wrappedafterbreak}{\char`\~}}%
        } 
        \newcommand*\Wrappedbreaksatpunct {%
            \lccode`\~`\.\lowercase{\def~}{\discretionary{\hbox{\char`\.}}{\Wrappedafterbreak}{\hbox{\char`\.}}}%
            \lccode`\~`\,\lowercase{\def~}{\discretionary{\hbox{\char`\,}}{\Wrappedafterbreak}{\hbox{\char`\,}}}%
            \lccode`\~`\;\lowercase{\def~}{\discretionary{\hbox{\char`\;}}{\Wrappedafterbreak}{\hbox{\char`\;}}}%
            \lccode`\~`\:\lowercase{\def~}{\discretionary{\hbox{\char`\:}}{\Wrappedafterbreak}{\hbox{\char`\:}}}%
            \lccode`\~`\?\lowercase{\def~}{\discretionary{\hbox{\char`\?}}{\Wrappedafterbreak}{\hbox{\char`\?}}}%
            \lccode`\~`\!\lowercase{\def~}{\discretionary{\hbox{\char`\!}}{\Wrappedafterbreak}{\hbox{\char`\!}}}%
            \lccode`\~`\/\lowercase{\def~}{\discretionary{\hbox{\char`\/}}{\Wrappedafterbreak}{\hbox{\char`\/}}}%
            \catcode`\.\active
            \catcode`\,\active 
            \catcode`\;\active
            \catcode`\:\active
            \catcode`\?\active
            \catcode`\!\active
            \catcode`\/\active 
            \lccode`\~`\~ 	
        }
    \makeatother

    \let\OriginalVerbatim=\Verbatim
    \makeatletter
    \renewcommand{\Verbatim}[1][1]{%
        \sbox\Wrappedcontinuationbox {\Wrappedcontinuationsymbol}%
        \sbox\Wrappedvisiblespacebox {\FV@SetupFont\Wrappedvisiblespace}%
        \def\FancyVerbFormatLine ##1{\hsize\linewidth
            \vtop{\raggedright\hyphenpenalty\z@\exhyphenpenalty\z@
                \doublehyphendemerits\z@\finalhyphendemerits\z@
                \strut ##1\strut}%
        }%
        \def\FV@Space {%
            \nobreak\hskip\z@ plus\fontdimen3\font minus\fontdimen4\font
            \discretionary{\copy\Wrappedvisiblespacebox}{\Wrappedafterbreak}
            {\kern\fontdimen2\font}%
        }%
        
        \Wrappedbreaksatspecials
        \OriginalVerbatim[#1,codes*=\Wrappedbreaksatpunct]%
    }
    \makeatother

    \definecolor{incolor}{HTML}{303F9F}
    \definecolor{outcolor}{HTML}{D84315}
    \definecolor{cellborder}{HTML}{CFCFCF}
    \definecolor{cellbackground}{HTML}{F7F7F7}
    
    \makeatletter
    \newcommand{\boxspacing}{\kern\kvtcb@left@rule\kern\kvtcb@boxsep}
    \makeatother

    \hypersetup{
      breaklinks=true,  
      colorlinks=true,
      urlcolor=urlcolor,
      linkcolor=linkcolor,
      citecolor=citecolor,
      }
    
\begin{document}
 \title{\bf \Large    Biconservative     Hypersurfaces in Euclidean Space  }
 	\title{\bf \fontsize{12}{12} \selectfont     BICONSERVATIVE HYPERSURFACES IN EUCLIDEAN SPACES }

	\author{ \fontsize{10}{10} \selectfont    HIBA BIBI,  MARC SORET AND MARINA VILLE}
	\date{\today}
\maketitle

		\begin{abstract} 
An isometric immersion  $X: \Sigma^n  \longrightarrow  \mathbb{E}^{n+1}$ is  biharmonic (BHH)  if
   $\Delta^2 X = 0$, i.e.,  if  $\Delta H =0$, where  $\Delta$ and  $H$ are
  the metric Laplacian   and the mean curvature  vector field of  $\Sigma^n$  respectively.  	
  More generally, biconservative hypersurfaces (BCHs) are isometric immersions for which only the tangential part of the biharmonic equation vanishes.
  BCHs are first described and  then constructed. 
  	\end{abstract}
\let\thefootnote\relax\footnotetext{ \hskip -.15 in  2024 { \it Mathematics Subject Classification}. Primary 53C40 58E20 Secondary  53C42  58E10\\
keywords: Hypersurfaces, Biharmonic, Biconservative,
Constant Mean Curvature, Minimal, Isoparametric submanifolds.}
\section{Introduction}
An isometric immersion in a Euclidean space, $X:\Sigma^n \longrightarrow \mathbb{E}^{N}$, is  biharmonic if
   $\Delta^2 X = 0$, i.e., if  $\Delta H =0$, where  $\Delta$ and  $H$ are   
  the metric Laplacian   and the mean curvature vector field   of  $\Sigma^n$ respectively.\\
The isometric immersion $X$  is a  critical point of the bienergy  functional (see \cite{EL}), where its 
 image is  called a biharmonic submanifold.
 B.Y. Chen conjectured in 
\cite{C0} that  a biharmonic submanifold  of  the Euclidean space must be minimal. 
In  this direction, a long-range argument   in  \cite{AM},  based on the maximum principle  at infinity,  proves B.Y. Chen's conjecture 
  for complete, topologically proper biharmonic hypersurfaces in Euclidean spaces without singularities. 
 In this paper, we  focus  on hypersurfaces,  and we study the  local  properties 
of both biharmonic hypersurfaces (BHHs) and    biconservative  hypersurfaces (BCHs). The BCHs  are less restrictive
 and are defined as critical points of the bienergy functional with respect to tangential deformations, i.e., satisying $(\Delta H)^\top=0$.
\\
 Constant mean curvature (CMC) hypersurfaces in a Euclidean space are trivially biconservative. Thus, we look for BCHs which are not CMCs, 
called proper BCHs. An exhaustive description   of all proper BCHs  in $\et$ was given  in \cite{MO1}, and then generalized  to the existence of 
a family of $SO(p)$-symmetric BCHs of $\mathbb{E}^p$   \cite{MOR}.  Also, in  \cite{MOR}, the authors  constructed  proper BCHs   
which are $SO(p)\times SO(q)$-symmetric    in $\mathbb{E}^{p+q}$,  and they   even  made a global study on these hypersurfaces.\\
 In this article, we start by studying  the local   behavior of  BCHs and BHHs.\\
Recall that, according to the definition  in \cite{T},
 a  submanifold is isoparametric if its normal bundle is flat and if the principal curvatures in the direction of any
   parallel normal vector field are constant. 
\begin{thm}\label{letheoreme1} A   proper BCH  $\Sigma^n $ in $ \mathbb{E}^{n+1}$ is  foliated by 
the level sets of the mean curvature function  $h$ of $\Sigma^n$ which  are isoparametric codimension $2$ submanifolds  of $\mathbb{E}^{n+1}$ in 
a neighbourhood  of a regular point of $h$.
 Furthermore, the symmetries of the leaves extend to symmetries of $\Sigma^n$.\\
Conversely,  if  $\rm{U_0}$  is a given nonempty  open set of an    isoparametric submanifold $ U_0\subset \mathbb{E}^{n+1}$
 of codimension 2 with nonzero mean curvature vector, then
 there exists  a nonempty  open subdomain $V_0\subset \rm{U_0}$   and a  
proper BCH   diffeomorphic to $V_0\times I$, $I$ open interval of $\mathbb{R}$, 
 obtained by a local normal flow  of submanifolds parallel to $\rm{U_0}$ in the normal bundle of $\rm{U_0}$ that preserves
  the symmetry group of $\rm{U_0}$. 
\end{thm}
Parts of Theorem \ref{letheoreme1} were already known (see for example   \cite{D}  and  \cite{HV}).\\
The method of construction of BCHs applies also to  minimal hypersurfaces:
 \begin{cor}
  Let $Y:U \mapsto \mathbb{E}^{n+1}$ be a codimension 2 open isoparametric submanifold of $\mathbb{E}^{n+1}$. 
  Then there is a nonempty  open subset $V\subset U$, a nonempty interval $I:=[0,\epsilon[$ and a minimal immersion $X: V\times I \mapsto \mathbb{E}^{n+1}$ such that 
  $X(V\times \{0\} )=Y(V)$.
  \end{cor}

 \section{Generalities on ${m}$-harmonic maps}
  \subsection{Notations} 

Since our results are of a local nature, we can, isometrically, identify $\Sigma^n$ with a local open subset  $\Omega$ of $\mathbb{R}^{n}$. 
 Let  
 $X:\Omega \subset \rn \longrightarrow \Sigma^n\subset \mathbb{E}^{n+1}$ be  a smooth  immersion,
 where  the image $\Sigma^n$ is an  open   hypersurface   of $\mathbb{E}^{n+1}$, with $\Omega$ endowed with 
 the pull-back metric.\\
Let   $\Delta_0 $ denote the flat Laplacian on $\Omega$: $$\Delta_0 = \sum_{i=1}^n \partial^2_{x_i}.$$ 
 By definition, a function $f$ defined on  $\Omega$ is harmonic if  
  $\Delta_0 f= 0,$ and  it is $m$-harmonic if $\Delta_0^m f=0,$ where
 $$\Delta_0^m := \overbrace{\Delta_0\circ\Delta_0\circ \cdots \circ \Delta_0}^m.$$
More generally,  a  map $X$ is  harmonic if $\Delta_0 X= 0$, i.e., $\Delta_0 X^i = 0$ for all $ i=1,\ldots,  n+1 $. Thus, $X$ is harmonic if and only if 
all the  coordinate functions 
that define  $X$ are harmonic. 
By extension, $X$  is $m$-harmonic if and only if each coordinate function of the position vector $X$ is  $m$-harmonic.

\begin{ex} When $\Sigma^n$ is a graph of a function $u$  over a hyperplane $\mathbb{R}^n$ in $\mathbb{R}^{n+1}$  given by the parametrization  
$$X: (x_1,\ldots,x_n) \mapsto (x_1,\ldots, x_n, u),
$$
 $\Delta_0 X=0$ if and only if the following $n+1$ linear PDEs are satisfied 
$$
\Delta_0 x_1=0,\ldots, \Delta_0 x_n=0 \quad \text{and} \quad \Delta_0 u =0.
$$
The first n-equations are always satisfied; hence, $X$ is harmonic   if and only if  $u$ is a harmonic function.
 These remarks extend to the $m$-harmonic case: $X$ is $m$-harmonic  if and only if  $\Delta_0^m u =0.$ \end{ex}
Suppose now that  $X=(X^1,\ldots, X^{n+1})$  is an isometric immersion.
Let    $g= g_{ij}$ be the induced metric on $\Omega$ and $|g|$ its determinant. The Laplacian of $X$ for this metric, which is  denoted by $\Delta$, 
is equal to the vector-valued function
 $$(\Delta X)^i=\Delta X^i = Tr(Hess(X^i)) \quad \forall i =1,\ldots, n+1.$$
 By definition,  $X$ is  then   an $m$-harmonic isometric immersion  if and only if  $\Delta^m X =0$. 
Equivalently, we say that  $\Sigma^n:=X(\Omega)$ is  an $m$-harmonic hypersurface of $\mathbb{R}^{n+1}$.
\\
Let us recall  that  the  Laplacian operator in terms of the induced metric: $"g"$ of a smooth function $f:\Omega \longrightarrow \mathbb{R}$,
using Einstein's summation convention, is given by 
\begin{equation}\label{lap}
 \Delta f:=  \frac{1}{\sqrt{|g|}} \left(\sqrt{|g|}g^{ij}f_{,i}\right)_{,j}  =0, \quad i,j = 1,\ldots, n,
 \end{equation}
  where $f_{,i}$ denotes the derivative of $f$ with respect  to each coordinate $x_i$, for all $i=1,\ldots, n$.
   \subsection{ Remarks on the equation of a biharmonic graph} 
 Let us find the PDE for a biharmonic graph. Before that, let us recall  the $1$-harmonic graph equation.
 Consider, as in the  previous example, the parametrization
$$X: (x_1,\ldots,x_n) \mapsto (x_1,\ldots,x_n,u),$$
 hence, $X_{,i}:= \frac{\partial X}{\partial  x_i}= (0,\ldots,0, 1,\ldots,0,u_{,i}).$ 
The induced metric of a graph on the domain  of $X$ is then given by a matrix whose coefficients are in terms of the  derivatives of $u$:
$$g_{ij} = \delta_{ij} +u_{,i}u_{,j}\quad {\rm and} \quad  g^{ij} =  \delta_{ij} -\frac{u_{,i}u_{,j}}{W^2},$$
where $(g^{ij}) $ is the inverse matrix of $(g_{ij})$ and 
$W:=\sqrt{1+|\nabla u |^2} = \sqrt{|g|}$.
The second fundamental  form and the shape operator of $X$ are,  respectively,  given by the following  matrices: 
$$II_{ij} = \frac{u_{,ij}}{W}\quad {\rm and}\quad A_{ij} = (^t g^{-1}II )_{ij}=\frac{u_{,ij}}{W}-\sum_{\alpha =1}^n\frac{u_{,i\alpha} u_{,\alpha} u_{,j}}{W^3},$$
where $i,j=1,\ldots, n$.
\\
For example, in dimension $2$, the metric and the second fundamental form are, respectively, given by
\begin{equation}
g =\left( 
\begin{array}{cc}
1+u_{,x}^2 & u_{,x}u_{,y} \\
u_{,x}u_{,y} & 1+u_{,y}^2
\end{array}
\right) \quad {\rm and}\quad
II =\frac{1}{W}\left( 
\begin{array}{cc}
u_{,xx} & u_{,xy} \\
u_{,xy} & u_{,yy}
\end{array}
\right).
\end{equation}\\
The mean curvature vector field of the graph 
$$
H= \frac{\Delta X}{n}=\frac{1}{n} (\Delta x_1,\ldots , \Delta x_n, \Delta u)
$$
decomposes into a horizontal part $$nH_{hor} = (\Delta x_1,\ldots ,\Delta x_n, 0)$$  and a vertical part
$$nH_{ver}=(0,\ldots , 0 , \Delta u).$$
A harmonic graph satisfies   $H_{hor}=0$  and $H_{ver}=0$. Using equation \eqref{lap},  for $H_{hor}=0$ and  $H_{ver}=0$, 
 we obtain, using Einstein's summation convention,    the following second-order quasi-linear PDEs:
\begin{equation}\label{lap2}  
\left\{
\begin{array}{ll}
nH_{hor}^l =\frac{1}{\sqrt{|g|}} \left(\sqrt{|g|}g^{kl}\right)_{,k}  =0,\quad  l,k=1,\ldots, n \\
nH_{ver}^{n+1} =\Delta u = g^{kl} u_{,kl} =0.
\end{array}
\right.
\end{equation}
If we substitute  the value of $g^{ij}$ in terms of $u$ in the second equation, we obtain the so-called minimal graph equation 
 which can also be expressed as:
$$\left(\frac{u_{,i}}{W}\right)_{,i}=0= {\rm div}\left( \frac{\nabla u}{\sqrt{1+|\nabla u|^2}}\right).$$
For example, in dimension $2$
$$ (1+u_{,y}^2)u_{,xx} -2u_{,x}u_{,y} u_{,xy} +(1+u_{,x}^2)u_{,yy}=0.$$
In other terms, a minimal surface is the image of an isometric immersion which is harmonic.\\
We, similarly, proceed  for  biharmonic graphs. 
 First, we deduce from \eqref{lap}  that the bilaplacian  applied to a function $f$ is:
\begin{equation}\label{lap3}
\begin{array}{c}
\Delta^2 f =g^{ij}\left( g^{kl} f_{,kl} + \frac{1}{\sqrt{|g|}} \left(\sqrt{|g|}g^{kl}\right)_{,k} f_{,l} \right)_{,ij} + \\
\frac{1}{\sqrt{|g|}} \left(\sqrt{|g|}g^{ij}\right)_{,i} \left( g^{kl} f_{,kl} + \frac{1}{\sqrt{|g|}} \left(\sqrt{|g|}g^{kl}\right)_{,k} f_{,l} \right)_{,j}.
 \end{array}
\end{equation}
The biharmonic equation, $\Delta^2X=\Delta H=0$, decomposes 
into a horizontal part:  $$n(\Delta H)_{hor} = (\Delta^2 x_1,\ldots ,\Delta^2 x_n, 0)$$  and a vertical part:
$$n(\Delta H)_{ver} =(0,\ldots , 0 , \Delta^2 u).$$ A straightforward computation yields the following  fourth order PDEs:
 \begin{equation}\label{lap5}
n(\Delta H)^m_{hor}  =g^{ij}\left(    \frac{1}{\sqrt{|g|}} \left(\sqrt{|g|}g^{km}\right)_{,k}  \right)_{,ij} +
 \frac{1}{\sqrt{|g|}} \left(\sqrt{|g|}g^{ij}\right)_{,i} \left(    \frac{1}{\sqrt{|g|}} \left(\sqrt{|g|}g^{km}\right)_{,k}   \right)_{,j}=0
\end{equation}
for $m = 1,\ldots, n$.
 Using \eqref{lap5}, we obtain
 \begin{equation}
\begin{array}{c}
n(\Delta H)_{ver}^{n+1} =  g^{ij}\left( (g^{kl} u_{,kl})_{,ij} + \frac{1}{\sqrt{|g|}} \left(\sqrt{|g|}g^{kl}\right)_{,k} u_{,lij} \right)+ \\
2g^{ij}\left(  \frac{1}{\sqrt{|g|}} \left(\sqrt{|g|}g^{kl}\right)_{,ik} u_{,lj} \right)+ \\
 \frac{1}{\sqrt{|g|}} \left(\sqrt{|g|}g^{ij}\right)_{,i} \left( g^{kl} u_{,klj} +\frac{1}{\sqrt{|g|}} \left(\sqrt{|g|}g^{kl}\right)_{,k} u_{,jl} \right)=0,
 \end{array}
\end{equation}
which is a semi-linear PDE  of order $4$ (its highest order term is  $g^{ij}g^{kl} u_{,klij}$). 

\section{     Biharmonic Hypersurfaces in $\mathbb{E}^{n+1}$}
Let us determine the PDEs that characterize  biharmonic hypersurfaces in $\mathbb{E}^{n+1}$
(for more general cases cf., for example,  \cite{C1}).\\
 Let $X$ be any  isometric immersion:
  $$X: \Omega \subset \mathbb{R}^n \longrightarrow X(\Omega)= \Sigma^n \subset \mathbb{E}^{n+1}.$$
  We denote the mean curvature  vector field of $\Sigma^n$  by $H$, its mean curvature function by $h$, the unit normal vector field by $N$, the shape operator by $A$ and the Laplacian operator by $\Delta$.  \\
It is easy to show that the  position vector field $X$ and the mean curvature vector  field $H$, where $H=hN$,  are related by the following Beltrami equation 
\begin{equation}\label{eq1}
\Delta X =nH.
\end{equation}
One can similarly show that 
\begin{equation}\label{eq 2}
\Delta N  =-|A|^2N - \nabla Tr A,
\end{equation}
where $Tr A= nh$. This equation decomposes the Laplacian of the unit normal vector field with respect to its  normal component ($|A|^2N$) and its tangential component ($\nabla Tr A$).
\\
  $X$ is biharmonic if and only if  
  $$
  \Delta H =\Delta	(hN) =0,
  $$ thus 
  \begin{equation}\label{Laplacian N}
    (\Delta h) N + h \Delta N + 2 \sum_{i=1}^n\nabla_{e_i} h\nabla_{e_i} N =  (\Delta h) N + h \Delta N + 2 \nabla_{\nabla h} N = 0.
  \end{equation}
   where $\nabla h$ denotes the gradient of $h$ on $\Sigma^n$.
 Now, from equations
  \eqref{eq1}, \eqref{eq 2} and \eqref{Laplacian N} we deduce that $\Sigma^n$ is biharmonic if and only if  it  satisfies the  following two  semi-linear PDEs of 
  orders $3$ and $4$, respectively:
  \begin{equation}\label{System3}
\left\{
\begin{array}{cc} 
 (\Delta H)^\top:= 2 \nabla_{\nabla h} N -nh\nabla h   &=0\\
 (\Delta H)^\perp :=  \Delta  h -|A|^2 h  &=0.\\
\end{array}
\right.\end{equation}

 The solutions of the  first equation ($(\Delta H)^\top=0$)   correspond
to  critical points  of the bienergy functional for  compact deformations that are tangent  to $\Sigma^n$. 
 \begin{notation}
A hypersurface satisfying $(\Delta H)^\top=0$  will be denoted by BCH which stands for "Biconservative Hypersurface",  and a hypersurface satisfying  both equations in \eqref{System3} will be denoted by BHH  which stands for "Biharmonic  Hypersurface".
\end{notation}
 \hypertarget{properDef}{
 \begin{notation}If $\Sigma^n$ has  constant mean curvature (CMC), then $\Sigma^n$ is a BCH. Furthermore, a CMC hypersurface satisfies the  normal equation in \eqref{System3} if and only if  $h=0$. 
 Thus, any minimal hypersurface $\Sigma^n$ is a BHH. We will later  denote by {\bf proper BCH }
 (respectively  {\bf proper BHH}) a BCH which is not CMC  \cite{MOR} (respectively  a BHH which is not minimal  \cite{MO}).
 In other words a BCH or BHH is proper if  the critical points of $h$ have empty interior.
\end{notation}
}
   \subsection{Biconservative  equation in a principal curvature frame}
  At a point $p\in \Sigma^n$, we have $s$ distinct eigenvalues and $s$ corresponding eigenspaces of the shape operator $A_N$. This defines $s$ plane distributions $T_i,$ where 
   $T_i = ker(A_N(p)-\lambda_i Id)$ and  $\lambda_i $ 
   is the corresponding eigenvalue, or principal curvature, of multiplicity $rk(T_i)=m_i$
    such that $T_p\Sigma^n  = \oplus_{i=1}^s T_i$, where  $\sum_{i=1}^s m_i = n$.
  We will select  an open subdomain  where  the multiplicities $m_i,\   i=1,\ldots, s,$ of the corresponding eigenvalues  are constant. 
We will obtain  a splitting of the tangent bundle $T\Sigma^n$  into  a family of integrable distributions
  $T_i\subset T\Sigma^n $ (see Appendix  \ref{cristoffel}). 
\\
In the following, we express the biharmonicity  in terms of a principal curvature frame.
    \begin{prop}\label{biTan1}
   The BCH equation (10.1) is equivalent to
   either $\nabla h = 0$ or $e_n:=  \frac{\nabla h }{|\nabla h| }$ is a principal curvature direction of $\Sigma^n$ of curvature 
   $-\frac{nh}{2}$, i.e., 
     \begin{equation}\label{leSysteme}  
D_{e_n}  N = -\lambda_n e_n = \frac{nh}{2} e_n,
\end{equation}
where $D$ is the covariant derivative of the ambient space and $\nabla$ the gradient on $\Sigma^n$.
   \end{prop}
   \begin{proof}
   If $\nabla h \neq 0$, then let $e_n :=  \frac{\nabla h}{|\nabla h|}$ then
  $$ \frac{(\Delta H)^\top}{|\nabla h| } = 2 \nabla_{e_n} N -nh e_n   = 0.$$
   \end{proof}

     \begin{defn}\label{class}
Let  $\Sigma^n$  be a BCH in $\mathbb{E}^{n+1}$. $\Sigma^n$ is  of rank  $k$  if $rk( T\Sigma^n({-\frac{nh}{2}}) )=k,$ 
where $T\Sigma^n({-\frac{nh}{2}})$ is the eigenspace of 
 the shape operator of eigenvalue ${-\frac{nh}{2}}$ in $T\Sigma^n$.
\end{defn}
Note that, $\nabla h$ is null  when projected onto  $T\Sigma^n({-\frac{nh}{2}}) ^\perp$.
As we will see in {\color{red} Lemma \ref{para1}}, the rank of $\Sigma^n$ is either $0$ or $1$.
\section{Reducible biconservative and biharmonic hypersurfaces}\label{reducible}
A canonical   way to construct  BCHs (or BHHs)  from 
BCHs (BHHs)  submanifolds of lower dimensions is to take    Riemannian products of them. 
 \begin{prop}
If $\Sigma_1^n\subset \mathbb{E}^{n+1}$ is a BCH  (respectively  a BHH) of rank $k$, 
then the Riemannian  product $\tilde\Sigma:= \Sigma_1^n\times \mathbb{E}^l \subset \mathbb{E}^{n+l+1}$
 is also a BCH (respectively  a BHH) of rank  $k$,  where $ \mathbb{E}^l$ is a  flat plane of dimension $l$.
 \end{prop}
\begin{proof} First, if $h\equiv 0$ (respectively  $h$ is constant), then $\Sigma^n$ is minimal  (respectively  a CMC)
and so is  the Riemann product  $\Sigma^n\times \mathbb{E}^l$ in $\mathbb{E}^{n+l+1}$ since 
$\Delta_{\Sigma^n \times  \mathbb{E}^l}  = \Delta_{\Sigma^n }  +  \Delta_{\mathbb{E}^l} $.\\
Second, suppose   that $h\not\equiv  c $ and $ h\neq 0$, then the Riemann product $\tilde \Sigma$ has a  mean curvature 
$$\tilde h = \frac{n}{n+l}h $$ and
$$rk(  T\tilde\Sigma ({-\frac{(n+l)\tilde h}{2}})) = rk( T\tilde\Sigma (-\frac{n h}{2}))=rk( T\Sigma^n (-\frac{nh}{2}))=k. $$
 Furthermore, $\nabla h$ is null on $T\Sigma^n(-\frac{n h}{2})^\perp \oplus  \mathbb{E}^l$.
Hence, it is biconservative  from Proposition \ref{biTan1} and of rank $k$.
Now, if $\Sigma^n$ is a BHH, then the  second equation of  the system \label{System2} holds also for $\tilde\Sigma$ since $|A|^2 = |\tilde A|^2$. \\
\end{proof}
BCHs (respectively  BHHs) which are Riemannian products  will be called {\sl reducible}. Thus, in the sequel, 
 we will  only  investigate {\sl irreducible} and \hyperlink{properDef}{\sl proper} BCHs.
\section{Geometry   of proper    BCHs in $\mathbb{E}^{n+1}$}
 We will show that   BCHs  are either  of rank $0$, i.e., are CMCs, or of rank $1$.
  In the latter case, they are  foliated by  level hypersurfaces defined by  the mean curvature function $h$.
   We  will study later the extrinsic geometry of these leaves and the geometry  of the integral curves which are
 orthogonal to them.
 Notice that Lemmas \ref{para1}, \ref{para3} and \ref{para4} are similar 
 to Lemmas 2.2, 2.3 and 2.5 of \cite{HV}, but our proofs differ.
This condition, in turn, allows us  to prove  lemmas \ref{paraODE} and \ref{para5} which were still  unknown. Also, the fact that  $rk(\lambda_n)=1$
in Lemma  \ref{para1} was already proved in \cite{D}.
 \subsection{  Foliation of  a BCH by the level surfaces of the mean curvature function}
\begin{lem}\label{para1} A BCH  $\Sigma^n\subset \mathbb{E}^{n+1}$ is either a CMC,  i.e.,  of rank $0$, or a proper BCH of rank  $1$. 
In the latter case,  $\Sigma^n$ is, 
in a neighbourhood of regular points of  the mean curvature function  $h$  of $\Sigma^n$,     foliated by the leaves, $U_t = h^{-1}(t)$,  which are the level hypersurfaces 
of $h$.
These $U_t$ are, locally, codimension 2 submanifolds of 
  $\mathbb{E}^{n+1}$  with  flat   normal bundle. The orthonormal frame $\{N,e_n\}$ of $NU_t$ is parallel, and the shape 
  operators along any two  parallel  normal directions commute, in particular,  $[A_N,A_{e_n}]=0$.
   The orthogonal curves to $U_t$, which are  integral curves of the vector field 
  $\nabla h$ on $\Sigma^n$ and which are called {\bf flow curves}, are curvature lines of principal curvature $\lambda_n = -\frac{nh}{2}$. Any other 
  principal curvature  satisfies $\lambda_i \neq -\frac{nh}{2},$ $  i =1,\ldots, n-1 $.
   \end{lem}
 \begin{proof}
First, if  the rank of the BCH is null,   then $h$ is trivially  constant.
Hence, the BCH is CMC,  and from the normal equation of  \eqref{System2},
a corresponding BHH 
 automatically satisfies $h=0 $, i.e.,  $\Sigma^n$ is minimal. \\ 
We now suppose   that the rank  $k$  of $\Sigma^n$ is at least two.
   In a neighbourhood  of regular points of $h$,  $\Sigma^n$ is foliated by hypersurfaces of $\Sigma^n$ which are codimension $2$ submanifolds of $\mathbb{E}^{n+1}$, and are given by
   $$
   U_t := \{p\in \Sigma^n : h(p) = t \} .
   $$
  As 
  $T\Sigma^n(-\frac{nh}{2})^\perp \subset TU_t\subset T\Sigma^n$, then 
  $TU_t^\perp \subset T\Sigma^n(-\frac{nh}{2}) \subset T\Sigma^n$.
  Thus, the normal curves to the foliation on $\Sigma^n$ are  integral curves of 
  $e_n:= \frac{ \nabla h}{|\nabla h|}$
and curvature lines of the principal curvature $\lambda_n = -\frac{nh}{2}$.\\
When the rank $k\geq 2$, the distribution defined by the $k$-planes, $T\Sigma^n(-\frac{nh}{2})$, 
  is integrable   into  a submanifold of dimension $k$ 
that has constant principal curvature $\lambda_n = -\frac{nh}{2}$  with respect to the normal  direction 
 $N$ (see Appendices \ref{cristoffel} and  \ref{spheres}). 
In particular,  $\lambda_n$ is constant along the  integral curves of $\nabla h$,
and so is $h$; but $h$ is also constant on the $U_t$, hence $h$ is constant everywhere which contradicts the hypothesis.
Thus, we conclude that  $rk(T\Sigma^n(-\frac{nh}{2}) )=1$.\\
Second,  let us show that the normal bundle $NU_t\subset \mathbb{E}^{n+1}$  is flat. Along the principal directions $\{e_i\}_{i=1}^{n-1},$ in $TU_t$, we have
$$\langle D_{e_i} N, e_n \rangle = -\langle \lambda_i e_i, e_n \rangle  =0 \quad {\rm and} \quad
\langle D_{e_i} e_n, N\rangle = -\langle  e_n, D_{e_i} N \rangle  =0,$$
where $D$ is the connection  in $\mathbb{E}^{n+1}$.
In particular,  the normal bundle is  flat and   $\nabla_{e_i}^\perp N = 
\nabla_{e_i}^\perp e_n =0$. Its  normal curvature is zero, and using the Ricci equation, we obtain
$$\langle R(X,Y)\xi , \tau \rangle = \langle [A_\xi,A_\tau ]X,Y \rangle =0,\ \forall X,Y \in TU_t, \ \forall \xi,\tau \in NU_t.$$
So the shape operators $A_N$ and $A_{e_n}$ commute. Hence $A_{e_n}$ 
is diagonalised in the same frame $\{ e_i\}_{i=1}^{n-1}$, i.e., there exist  principal curvatures 
 $\{\mu_i\}_{i=1}^{n-1} $ in the $e_n$ normal direction $e_n$
 such that 
  \begin{equation} \label{mu}
 \nabla_{e_i} e_n = -\mu_i e_i, \quad i = 1, \ldots, n-1.
 \end{equation} 
\end{proof}
\subsection{Curvature lines of hypersurfaces of the Euclidean space}\label{holonomic}
For more details on principal curvatures  we refer, for example, 
to \cite{DT}.
 At each point $p\in \Sigma^n$,  the tangent space $T_p\Sigma^n$ decomposes into  an orthogonal sum 
of the 
 eigenspaces $T_i$ of  the  symmetric shape operator $A_N(p)$ corresponding to    distinct eigenvalues   $\lambda_i, \ i =1,\ldots, s$, i.e.,  $
 T_p\Sigma^n = \bigoplus_{i=1}^s T_i
 $, where $N$ is a local unit normal vector field of the normal bundle $N\Sigma^n$.  \\
Let us consider a neighborhood  of a  point $p$, where  the planes $\{T_i\}_{i=1}^s$ define a smooth distribution.
From the Codazzi  equations (cf. Appendix  \ref{cristoffel})  one shows that the distribution 
$T_i$ is integrable for each $i=1,\ldots, s$.
 Then, again, using  the Codazzi equations, 
  the leaves of the distribution are totally umbilical. Recall that a submanifold is totally umbilical if 
  the second fundamental form is equal to  $(D_XY)^\perp  =\langle X,Y\rangle H$, 
where $H$ is the mean curvature vector field of  the submanifold.
 When $rk(T_i) \geq 2$, the leaves of the distribution $T_i$ are  $rk(T_i)$-spheres  
 of $\mathbb{E}^{n+1}$ (see for example Appendix  \ref{spheres}).
\subsection{Flow curves}
\begin{lem}\label{para3} 
Let $\Sigma^n$ be a   proper  BCH.  The h-level    hypersurfaces     $U_{t}  \subset \Sigma^n$ form a Riemannian foliation around regular points.
The orthogonal flow curves  are planar  in $NU_{t}$, i.e., for each $p\in U_{t}$,    $N_pU_{t}\cap \Sigma^n$ is a curvature line of principal curvature 
 $\lambda_n$ and    a  geodesic of $\Sigma^n$. 
\end{lem}
\begin{proof}\phantom{x} 
First, consider $\gamma$ an integral curve of the unit vector field $  e_n := \frac{ \nabla h}{|\nabla h|}$.\\
  $\gamma$ is a   geodesic of $\Sigma^n$. Indeed,
 by the Codazzi equations, and since $\lambda_n$ is constant on the leaves $U_{t}$, we have (cf. Appendix \ref{cristoffel}): 
\begin{equation} \label{parallele}
 \langle\nabla_{e_n} e_i ,e_{n} \rangle = \frac{e_i\cdot \lambda_n}{\lambda_i-\lambda_n} =0.
 \end{equation}
 Hence,
 $$\langle \nabla_{e_n}e_n,e_i \rangle = - \langle e_n ,\nabla_{e_n}e_i \rangle=0,\quad  i=1,\ldots, n.$$
As  $\langle \nabla_{e_n}e_n,e_n \rangle=0$, $\gamma$ is a geodesic.\\
  Second,  geodesics that are curvature lines must be planar. Indeed,\\
 let $T$ be the  tangent vector field to the curve and $N$ denotes the unit normal vector field to $\Sigma^n$,
$$\frac{D}{ds}(T\wedge N) = \frac{D}{ds}T\wedge N + T\wedge\frac{D}{ds}N.$$
As $\gamma$ is a geodesic of $\Sigma^n$,   $\frac{D}{ds}T =\kappa N$,
and as $ \frac{D}{ds} N=-\lambda_n T,$ we obtain
 \begin{equation}\label{planConstant}
 \frac{D}{ds}(T\wedge N) = 0.
 \end{equation}
 Therefore, the integral curves of $\nabla h$  are planar and contained in the fixed  plane defined by $T_0\wedge N_0$.
 We will  show in the next lemma that all the flow  curves are congruent.\\
   \end{proof}
\subsection{ A parametrization of    BCHs}
From the former section, we know that a BCH,  on a regular subdomain,  is locally a
 Riemannian foliation of  the level surfaces of the mean curvature  function $h$ on 
$\Sigma^n$, and is also foliated  by the orthogonal geodesics that are integral curves of 
the vector field $e_n$. 
\subsubsection{Curvature notations} \label{curvatureNotations}

On an h-level set $U_0$, we consider the parallel frame $\{e_{0n}(p),N_0(p)\}_{p\in U_0}$. As $U_0$ is a submanifold of codimension 2 in $\mathbb{E}^{n+1}$, we will denote by  $A_{N_0}(p)$ (respectively  $A_{e_{0n}}(p)$) its shape operator in the $N_0$  direction at $p$  (respectively  in the $e_{0n}$ direction at $p$). The corresponding eigenvalues of $A_{N_0}(p)$  (respectively  $A_{e_{0n}}(p)$) will be denoted by $\lambda_{0i}(p)$, $i \in \{1,\ldots, n-1\}$, (respectively  by $\mu_{0i}$, $i \in \{1,\ldots, n-1\}$).
The mean curvature vector field of $U_0$ at $p$  is given by  $$(n-2) H_{U_0}(p) = M_1(p) e_{0n}(p) +
 \Lambda_1(p) N_{0}(p):= \sum_{i=1}^{n-1} \mu_{0i}(p)e_{0n}(p) + \sum_{i=1}^{n-1} \lambda_{0i}(p)N_{0}(p).$$
For simplicity, the dependence on the point $p\in U_0$, the shape operators and the
associated eigenvalues  on $U_0$ will be implicit.
\subsubsection{Parametrization} \label{parametrization}
 \begin{lem}\label{para4} 
Let $\Sigma^n$ be a   proper BCH.
Let $N_0$ be the unit normal vector field of $\Sigma^n$ restricted to  a level set 
 $U_{0}$,  and  let  $e_{n} = \frac{\nabla h }{|\nabla h|}$. Then, there exists in a neighborhood of  a point $p\in U_{0}$  a local parametrization of  $\Sigma^n$, 
  $x_1,\ldots,x_{n-1}, t$, such that 
\begin{equation}\label{surRev}
X(x_1,\ldots,x_{n-1},t) = Y(x_1,\ldots, x_{n-1}) + \alpha(t)N_0(x_1,\ldots,x_{n-1}) + t e_{0n}(x_1,\ldots, x_{n-1}),
\end{equation}
where $Y$ parametrizes $U_{0}$, and $\alpha$ is a smooth  function such that 
 the immersion $m_p(t):=t e_{0n} +\alpha(t)N_0$ parametrizes the flow  curve starting at $p$.
In other terms, one can also write 
\begin{equation}\label{surRev2}
X(x_1,\ldots,t) = Y(x_1,\ldots, x_{n-1}) + \zeta(x_1,\cdots x_{n-1},t ) 
\end{equation}
where $\zeta$ is a normal parallel vector field in the normal space of  $U_0$. Thus, $\Sigma^n$ is  generated  by a flow of  normal vector fields
$\zeta_{,t} = (\alpha'(t) -t\alpha_n)N_0 + (1-\alpha(t)\lambda_n)e_{0n}$,  generating parallel submanifolds $U_t$ of codimension 2.
\end{lem}
\begin{rem} Note that $\zeta(p,t) = m_p(t)$.
\end{rem}
\begin{proof}
 By Lemma \ref{para3}, there is a local parametrization of $\Sigma^n$  
\begin{equation}
\begin{array}{cc}
 X:&\Omega \subset \mathbb{E}^n\longrightarrow \Sigma^n \subset \mathbb{E}^{n+1} \\
&(x_1,\ldots,x_n) \mapsto X(x_1,\ldots,x_n)
\end{array}
\end{equation}
such that the level sets $U_{t}$ are parametrized by $\{ (x_1,\ldots, x_{n-1},t)  \in \Omega \  \}$  with $t$ fixed, and the flow curve passing through some point 
$p=X (x_1,\ldots, x_{n-1},t_0) \in U_{t_0}$
is  parametrized by $\{ (x_1,\ldots, x_{n-1},t) \in \Omega\}$ with $(x_1,\ldots, x_{n-1})$ fixed.
 Let  $U_{0}$ be  parametrized by   
$$
Y(x_1,\ldots, x_{n-1}) := X(x_1,\ldots, x_{n-1},0).
$$
Again, using lemma \ref{para3}, 
for each point $p = Y(x_1,\ldots, x_{n-1})$, the integral curve of $\nabla h$ passing through $p$, which is a  curvature line parametrized by $t$,
is a planar curve  in the affine plane $P_p$ passing through $p$ and generated by the normal at $p$ to $\Sigma^n$, $N_0$,  and  by 
$e_{0n} =\frac{\nabla h }{|\nabla h|}$.
Locally, near $U_{0}$, this curvature line is a graph in $P_p$ over the $\mathbb{R}_{e_{0n}}$-axis. 
Thus,  there exist a function $\alpha(x_1,\ldots,x_{n-1}, t) $   such that:
 \begin{equation}\label{surRev2}
X(x_1,\ldots,x_{n-1},t) = Y(x_1,\ldots, x_{n-1}) + \alpha(x_1,x_2,\ldots,x_{n-1},t)N_0(x_1,\ldots,x_{n-1}) +t  e_{0n}(x_1,\ldots, x_{n-1}).
\end{equation}
 Since  the $X_{,i}$ are tangent to the level sets, 
\begin{equation}   \label{product}
\langle X_{,i}, X_{,t}\rangle =0,\quad  {\rm for\ all}\   i \in \{1,\ldots, n-1\}.
\end{equation} 
Choosing  $(x_1,\ldots, x_{n-1}) $ such that, at a point $p\in U_0$
 the tangent vectors $Y_{,i} $ are principal curvature directions, we obtain along the flow curve starting at $p$:
\begin{equation}\label{X}
X_{,i}  =( 1 -  \alpha\lambda_{0i} )  Y_{,i} 
 +t e_{0n,i}+ \alpha_{,i}N_0 =( 1 -  \alpha \lambda_{0i}-t\mu_{0i} )  Y_{,i} 
 + \alpha_{,i}N_0 , \quad i\in \{1,\ldots, n-1\}, 
\end{equation}
 and
 \begin{equation}\label{ALPHA}
  X_{,t}= \alpha_{,t} N_0 +  e_{0n}.
\end{equation}
From equations \eqref{product}, \eqref{X} 
 and \eqref{ALPHA} one obtains $\alpha_{,t}\cdot \alpha_{,i} =0$. If $\alpha_{,t}$ is null, then the $t$-curves are straight lines, i.e., $\lambda_n =0$. 
This implies that $\lambda_n=\frac{-nh}{2}\equiv 0$, i.e., $\Sigma^n$ is minimal,  which contradicts the fact that $\Sigma^n$ is proper  BHH.
Consequently, $\alpha_{,i} = 0$.
Hence, $\alpha$ is a function of $t$ only.
\end{proof}
\subsection{ Flow  curves of BCH}\label{ODE0}
As we have seen in the previous  lemma the flow curves are graphs of the  function $\alpha(t)$.
Now, we show that $\alpha$ satisfies a second order  integrable ODE. This will be deduced in the next section.
\subsubsection{Evolution of the principal curvatures along the flow curves }\label{evolutionPC}
Given a point  $p\in U_0$. The flow curve starting at $p$ lies in the affine plane 
$p+ span(e_{0n},N_0)$ which is  the normal plane to $U_0$ at $p$  
and is given by the parametrization:  $m_p(t)= te_{0n} + \alpha(t)N_{0}$.
 We compute the principal curvatures 
of the codimension 2 leaves $U_t$, along the curve in terms of a parameter $t$ along the curve $m_p$. 
We  will show that the  directions of the different normal curvature vector fields  of $U_t$
are fixed in the normal plane, but their lengths are not.
\begin{lem} \label{courbureVar1} With the notations  of section \ref{parametrization},
for  $t\in I$ and  $ i\in \{1,\ldots , n-1\}$
\begin{equation}\label{variation}
\lambda_i (t)=(\lambda_{0i}-\alpha'(t)\mu_{0i})\frac{\gamma}{\beta_i}(t)  \quad \text{and} \quad
\mu_i(t) =(\mu_{0i}+\alpha'(t)\lambda_{0i})\frac{\gamma}{\beta_i} (t),
\end{equation}
where $ \gamma(t) = \frac{1}{\sqrt{1+\alpha'^2(t)}} \ {\rm and} \ \beta_i(t) = 1 -  \alpha(t)\lambda_{0i} -t\mu_{0i}. $
\end{lem}
\begin{proof}
From Lemma  \ref{para4} equation \eqref{X}, 
 as $\alpha$ is a function of $t$ only, we  deduce  that 
  \begin{equation}\label{metriqueEvolution}
  \begin{array}{ll}
        X_{,i}&   =( 1 -  \alpha\lambda_{0i} -t\mu_{0i})  Y_{,i}  := \beta_i Y_{,i} 
    \quad i\in \{1,\ldots, n-1\},
    \\
      X_{,t}&= \alpha' N_0 + e_{0n}.  
  \end{array}
  \end{equation}
  The coefficients of the metric  of $\Sigma^n$  along the flow curve are   thus given by
   $$v_i ^2=\langle X_{,i},X_{,i}\rangle= v^2_{0i}( 1 -  \alpha(t)\lambda_{0i} -t\mu_{0i})^2,\quad {\rm if}\  i\neq n$$
 and 
 $$v^2_n = 1+\alpha'^2.$$
Hence,
\begin{equation} 
 g_{ii} = \beta_i^2 v^2_{oi}, \  i\in \{1,\ldots, n-1\} \ \text{and} \  g_{nn} = 1+\alpha'^2=v_n^2.
 \end{equation}
 where $v_{0i}^2=\langle Y_i,Y_i\rangle$.\\
  Along  the flow curve defined by  $m_p$, the normal vector field to $\Sigma^n$, $N$, and the  tangent vector field  to $m_p$, $e_n$,   are given by:
   
 \begin{equation}\label{unitNormal}
N=(N_0 -\alpha'  e_{0n})\frac{1}{\sqrt{1+\alpha'^2}}\ \text{and} \ e_n=(e_{0n} +\alpha'  N_{0})\frac{1}{\sqrt{1+\alpha'^2}}.
\end{equation}
 By  the  definition of the principal curvatures $\lambda_{0i}$ of $U_0$, we have:
$$N_{0,i} =-\lambda_{0i}Y_{,i} \ \  i\in\{1,\ldots, n-1\}.$$
 Plugging  these identities   into the derivatives of \eqref{unitNormal} 
  with respect to  $\frac{\partial}{\partial x_i}$, i.e., in the $e_i$ direction, we get
 \begin{equation} \label{variationDeN}
N_{,i} =N_{0,i} \frac{1}{\sqrt{1+\alpha'^2}} +\alpha'\mu_{0i} \frac{1}{\sqrt{1+\alpha'^2}}Y_{,i}=
-\left(  \lambda_{0i}  -  \alpha'\mu_{0i}    \right) Y_{,i}     \frac{1}{\sqrt{1+\alpha'^2}}.
\end{equation}
 From equations  \eqref{metriqueEvolution} and  \eqref{variationDeN}, we have $N_{,i} = -\lambda_i X_{,i} =-\lambda_i Y_{,i}\beta_i$, and  we 
obtain  the  variation of the principal curvatures of the level sets  along the flow   starting at $p$,  in terms of $t$:
\begin{equation}\label{variationDeLambda}
\lambda_i (t)=(\lambda_{0i}-\alpha'(t)\mu_{0i})\frac{\gamma}{\beta_i}(t) ,\quad 
 i\in \{1,\ldots , n-1\},
\end{equation}
\begin{equation}
{\rm where} \  \gamma(t) = \frac{1}{\sqrt{1+\alpha'^2(t)}}\ {\rm and } \  \beta_i(t) = 1 -  \alpha(t)\lambda_{0i} -t\mu_{0i}.
\end{equation}
Similarly, one computes   the principal curvatures of  $A_{e_n}$   along the flow. From \eqref{unitNormal}, we obtain 
\begin{equation} \label{variationDeEn}
e_{n,i} =
\left(  -\mu_{0i}  -   \alpha'\lambda_{0i}    \right) Y_{,i}     \frac{1}{\sqrt{1+\alpha'^2}}
\end{equation}
from which we deduce that
\begin{equation}\label{variationDeMu}
\mu_i(t) =(\mu_{0i}+\alpha'(t)\lambda_{0i})\frac{\gamma}{\beta_i}(t) ,\quad 
 i\in \{1,\ldots , n-1\}.
\end{equation}
\end{proof}
There is a  more concise and geometric expression of  equations \eqref{variation}.
First, we introduce the normal principal curvature vector  fields,  $n_i := \mu_i e_n + \lambda_i N$, in the 
     parallel normal bundle frame $\{e_n,N\}$.
Recall that the tangent space, $T_pU_t$,
 can be decomposed into eigenspaces $T_i$ common to all the shape operators
(its  normal bundle is flat), and for each $e_i$, there is a normal vector field $n_i$ 
such that $A_\zeta(e_i) =\langle n_i,\zeta\rangle e_i$. Second, we can also use  complex coordinates
 and  the arc-length parameter,   $s$,  to parametrize the  flow curve. Denoting by  $T$ the tangent vector field, $T(s)= e_n(s)$, of the flow  curve and by $X(s)$ the position vector, we have
\begin{equation}\label{theta}
\left\{
\begin{array}{ll}
T(s)= \gamma(1+i\alpha'(t(s))) &= \cos\theta(s) + i\sin\theta(s)= \cos\theta(s) e_n +\sin\theta(s) N\\
 X(s)= t(s)+\alpha(s) i &= \int_0^s e^{i\theta(u)} du = \int_0^s\cos\theta(u) du + i  \int_0^s\sin\theta(u) du,
 \end{array}
 \right.
\end{equation}
where $\theta$ is the angle between $T$ and $e_n$.
And for each $T_i$, we have
 \begin{equation}
\left\{
\begin{array}{ll}
{\rm normal\ curvature\ vector:} n_i= \mu_i e_n +\lambda_i N&= \mu_i +i\lambda_i   \\
{\rm normal\ curvature:}\nu_i := |n_i| &= \sqrt{\lambda_i^2+\mu_i^2}\\
{\rm radius\ of\  curvature:}\rho_i  &:=\frac{1}{\nu_i}.
 \end{array}
 \right.
\end{equation}
In particular,
\begin{equation} \label{variationDesCourbures}
\left(
\begin{array}{c}
\mu_i \\
 \lambda_i 
\end{array}
\right)
= \frac{1}{\beta_i}
\left(
\begin{array}{c}
\mu_{0i }\\
\lambda_{0i }
\end{array}
\right).
 \end{equation}
In other words, the principal normal vectors have invariant directions in the moving adapted  frame  $\{e_n,N\}$. However, they are dilated by a factor $ \frac{1}{\beta_i}$, 
where $\beta_i / |n_{0i}|$  is the distance from the point $p(t)$, of position vector 
$m_p(t)=(t,\alpha(t)) $, to the focal line $l_i(X)$ of equation:
$1 - \langle X,n_{0i}(t)\rangle =0$ (see for example \cite{T}).
\begin{lem}\label{courbureVar2} With the notations introduced before, let  $J= :e^{i\pi/2}$, and let  $dist(p,l_{0i}) $ be
 the signed distance 
from the point $p$ to the focal line $l_{0i}: \{ x \in \mathbb{R}^2\ : \ \langle x,n_{0i}\rangle = 1\}$ 
corresponding to the normal curvature vector $n_{0i}$ at $p_0$:
\begin{equation}
\begin{array}{llcl}
{\sl eigenvalues\  of\ } A_{N:} & \lambda_i(\zeta) =& \langle  \frac{ n_{0i}}{1-\langle n_{0i},\zeta\rangle}, N\rangle 
&= -\langle    \frac{ J \frac{n_{0i}}{| n_{0i} |} }{dist(p,l_{0i}) }, T \rangle \\
 & &\\
{\sl eigenvalues\  of\ } A_{e_n}: & \mu_i(\zeta) = &\langle  \frac{ n_{0i}}{1-\langle n_{0i},\zeta\rangle}, T\rangle &=
 \langle    \frac{ J \frac{n_{0i}}{| n_{0i} |} }{dist(p,l_{0i}) }, N \rangle \\
 & &\\
{\sl normal\  curvature\  vectors}:& n_i(\zeta) =& \frac{n_{0i}}{1-\langle n_{0i},\zeta\rangle} &= \frac{\frac{n_{0i}}{|n_{0i}|} }{\rho_i-\langle \frac{n_{0i}}{|n_{0i}|} ,\zeta\rangle} =
 \frac{\frac{n_{0i}}{|n_{0i}|} }{dist(\zeta,l_{0i})} \\
  & &\\
 {\sl mean\  curvature\  of\  U_t}:& H_{U_t}(\zeta) = &\frac{1}{n-1}\sum_{i=1}^{n-1} n_i  &=  \frac{1}{n-1} \sum_{i=1}^{n-1} \frac{n_{0i}}{1-\langle n_{0i},\zeta\rangle},
 \end{array}
 \end{equation}
 where $\zeta$ is the deformation vector field defined in Lemma \ref{para4}.
 \end{lem}
\begin{proof} 
As $\gamma (\lambda_{0i}-\alpha'\mu_{0i})= \langle n_{0i}, N\rangle,$ $\beta_i = 1-\langle n_{0i},\zeta\rangle$, $J^2=Id$ and $JT=N$, we obtain
\begin{equation}\label{variationDeLambda}
\begin{array}{lll}
\lambda_i &=(\lambda_{0i}-\alpha'(t)\mu_{0i})\frac{\gamma}{\beta_i}  
 &= \frac{\langle n_{0i}, N\rangle }{1-\langle n_{0i},\zeta\rangle} = -\langle    \frac{ J \frac{n_{0i}}{| n_{0i} |} }{dist(p,l_{0i}) }, T \rangle . \\ 
 \hskip -1,5 in {\rm Similarly,}& \phantom{x}\hfill\\
 \mu_i &=(\mu_{0i}+\alpha'(t)\lambda_{0i})\frac{\gamma}{\beta_i} 
 &= \frac{\langle n_{0i}, T\rangle }{1-\langle n_{0i},\zeta\rangle} =
 \langle    \frac{ J \frac{n_{0i}}{| n_{0i} |} }{dist(p,l_{0i}) }, N \rangle.
 \end{array}
\end{equation}
Note that $U_t$ is parallel to $U_0$  since $U_t = U_0 + \zeta$, where 
$\zeta= te_{0n} + \alpha(t) N_0 $. It follows directly that $n_i = \frac{n_{0i}}{1-\langle n_{0i},\zeta\rangle} $.
\end{proof} 
Third, we give below the variation of certain symmetric functions of the curvatures in terms of the coordinates 
in $span(e_{0n},N_0)$ and in terms of the normal curvature vectors $n_{0i}$.
We will use the following notations in the sequel for symmetric functions of the principal curvatures of the level sets.
\begin{defn} \label{defSym} Let 
$$\Lambda_k := \sum_{i=1}^{n-1} \lambda_i^k,\quad  M_k := 
\sum_{i=1}^{n-1} \mu_i^k\quad \text{and} \quad  N_{kl}  := \sum_{i=1}^{n-1} \lambda_i^k\mu_i^l.$$
\end{defn}
\begin{lem} 
We can expand the coefficients of the former  ODE's using either cartesian coordinates,  polar coordinates or using  the arc-length.
This yields the following different expressions for the coefficients $\mu$ and $\Lambda_2$.
First, let us introduce the series in $u$, $\pi_k(u) := \left(\frac{1}{1-u}\right)^{(k-1)}$, which is 
the formal $(k-1)$ derivative of $\frac{1}{1-u}$.\\
For $M_1$, we have
 \begin{equation}
M_1  :=
\sum_{i=1}^{n-1}\mu_{i=1} =\sum_{i=1}^{n-1}(\mu_{0i}+\alpha'\lambda_{0i})\frac{\gamma}{\beta_i} =
\frac{1}{\sqrt{1+\alpha'^2}} \sum_{i=1}^{n-1} \frac{\mu_{0i}+\alpha'\lambda_{0i}}{ 1 -\alpha\lambda_{0i}-t\mu_{0i}},
\end{equation}
 or
 \begin{align*}
    M_1 & = \frac{d}{ds}\left(  \sum_{i=1}^{n-1}\frac{1}{dist(p,l_{0i})  }\right) =
\frac{d}{ds}\left(  \sum_{i=1}^{n-1} \frac{|n_{0i}|}{ 1-\langle n_{0i},X\rangle}\right) 
\\
& = \frac{d}{ds}\left(  \sum_{i=1}^{n-1}  |n_{0i}|\pi_1(\langle  n_{0i},X \rangle  )   \right) = 
 \sum_{i=1}^{n-1} |n_{0i}|\langle n_{0i},T    \rangle\pi_2(\langle  n_{0i},X \rangle  ). 
 \end{align*}
  For $\Lambda_2$, we have 
 $$\Lambda_2:= \sum_{i=1}^{n-1} \lambda_i^2=  \frac{1}{1+\alpha'^2}  \sum _{i=1}^{n-1} 
  \frac{(\lambda_{0i}-\alpha'\mu_{0i})^2}{(1 -\alpha\lambda_{0i}-t\mu_{0i})^2 }, 
 $$
  or in terms of $e_{0i}:=\frac{n_{0i}}{|n_{0i}|}$
  $$\Lambda_2=    \sum _{i=1}^{n-1} 
 \left(  \frac{\langle n_{0i},N\rangle}{1 -\langle n_{0i},X\rangle } \right)^2=
   \sum _{i=1}^{n-1} 
 \left(  \frac{\langle e_{0i},N\rangle}{\rho_i -\langle e_{0i},X\rangle } \right)^2
  =    \sum _{i=1}^{n-1} 
 \left(  \frac{\langle e_{0i},N\rangle}{ dist(p,l_{0i}) } \right)^2:= \sum_{i=1}^{n-1}\tilde \pi_2(\langle n_{0i},X\rangle).$$
\end{lem}

\subsubsection{Equation of the flow curve of BCHs}
We use the notations of subsection \ref{curvatureNotations}.
\begin{prop}\label{paraODE}
A flow curve starting at $p\in U_0$ is a curve $\zeta(p,t) =te_{0n}(p)+\alpha(t)N_0(p), t\in I,$ in the  fixed normal  affine 
 plane
$p+NU_0(p) = p + span( e_{0n},N_0)$. The function  $\alpha$ satisfies the  second order  ODE:
\begin{equation}\label{lambda100}
  \frac{\alpha''}{1+\alpha'^2}(t) = - \frac{1}{3}  \sum _{i=1}^{n-1} 
  \frac{\lambda_{0i}-\alpha'(t)\mu_{0i}}{1 -\alpha(t)\lambda_{0i}-t\mu_{0i} }        
  \end{equation} 
  with initial conditions $\alpha(0)=\alpha'(0)=0$.\\
  Equivalently, in terms of $\zeta$ and its arc-length $s$, we have
\begin{equation}\label{lambda120}
  \left\langle \zeta_{,ss}+ \frac{n-1}{3} H_U(\zeta) ,J\zeta_{,s}\right\rangle = 0  \end{equation} 
  with initial conditions $\zeta(0)=p, \zeta_{,s}(0)=e_{0n}$.
\end{prop}
\begin{proof}
From Lemma \ref{courbureVar1}, we have 
\begin{equation}\label{lambda10}
\Lambda_1:=\sum_{i=1}^{n-1} \lambda_i  =  \frac{1}{\sqrt{1+\alpha'^2}}  \sum _{i=1}^{n-1} 
  \frac{\lambda_{0i}-\alpha'\mu_{0i}}{1 -\alpha\lambda_{0i}-t\mu_{0i} } . 
\end{equation}
As   $nh+2\lambda_n=0,$ or equivalently, as  $\sum_{i=1}^{n-1} \lambda_i +3\lambda_n = \Lambda_1 +3\lambda_n =0$, we get
\begin{equation}\label{lambda1}
\lambda_n = - \frac{1}{3\sqrt{1+\alpha'^2}}  \sum _{i=1}^{n-1} 
  \frac{\lambda_{0i}-\alpha'\mu_{0i}}{1 -\alpha\lambda_{0i}-t\mu_{0i} } . 
\end{equation} 
 Differentiating   $N$  in the $\frac{\partial}{\partial t}  $  direction gives  
$$N_{,t} =  (\frac{1}{\sqrt{1+\alpha'^2}})' N_0 -  e_{0n}(\frac{\alpha'}{\sqrt{1+\alpha'^2}})' = -\lambda_n X_{,t} =-\lambda_n ( \alpha' N_0 + e_{0n}).$$
 Thus, we simply obtain the standard expression    for the curvature of the planar  flow curve  in terms of $\alpha$
 \begin{equation}\label{lambda2}
 \lambda_n =  \frac{\alpha''}{(1+\alpha'^2)^{3/2}}  \left(= \frac{d\theta}{ds}\right),
\end{equation}
where $\theta$ is defined in equations \eqref{theta}.
Plugging equation  \eqref{lambda2} into equation \eqref{lambda10} gives equation  \eqref{lambda100}.
Equivalently,  from Lemma \ref{courbureVar2}, equations   \eqref{lambda10}  and   \eqref{lambda100} can be also written as
\begin{equation}\label{courbeIntegraleTer}
\begin{array}{ll}
\Lambda_1:=\sum_{i=1}^{n-1} \lambda_i  &= \left\langle     \sum _{i=1}^{n-1} 
  \frac{n_{0i} }{1-\langle n_{0i}, \zeta\rangle}, J\zeta_{,s} \right\rangle  \\
  \quad {\rm and}\quad & \\
   \left\langle\zeta_{,ss}+ \frac{n-1}{3} H_U ,J\zeta_{,s}\right\rangle &= 0.
 \end{array}
 \end{equation}
Since  $\zeta_{,s}$ is the tangent vector field to the flow curve,   $J\zeta_{,s}$ is the normal to the curve, and 
 $H_{U_{t}}(\zeta)$ is  the mean curvature vector field at $\zeta(t)$ given in Lemma \ref{courbureVar2}.
\end{proof}
$$
\begin{tikzpicture}[scale = 1]
		\draw [thick,color = red ] (-3,-2.32)--(3,3.68);
		\draw [thick,color = green ] (3,.9)--(1.5,2.2);

		 \draw[color=blue,domain = 0:5,very thick]    plot (\x,\x^2/10)             node[right] {$\alpha$};
		\draw [dashed] (0,0) -- (0,6);
		\draw [dashed] (0,0) -- (8,0);
		\draw [very thick,->] (0,0)--(0,2);
		\draw [very thick,->] (0,0)--(2,0);
		\draw [very thick,->] (0,0)--(-2,2);
		\draw[-stealth, thick] (3,.9) -- (4,1.6);
		\draw [very thick] node at (-2,1.5) {$n_{0i}$};
		\draw [very thick] node at (.45,1.9) {$N_0$};
		\draw [very thick] node at (0,-.3) {O};
		\draw [very thick] node at (3,.6) {P};
		\draw [very thick] node at (2.4,1.6) {\tiny ${\rho_i-\langle \frac{n_{0i}}{|n_{0i}|},\vec{OP}(t)}\rangle$};
		\draw [very thick] node at (2.7,3) {$l_i$};
		\draw [very thick] node at (1.7,-.3) {$e_{0n}$};
		\draw [very thick] node at (3.7,1) {T};
		\foreach \Point in {(0,0),(3,.9),(1.5,2.20) }{ \node at \Point {\textbullet};}
\end{tikzpicture}$$
\begin{center}
{\small Fig 1.  A Flow curve and a focal line }
\end{center}

\subsection{ The level  sets of the mean curvature function $h$ on a BCH are isoparametric }\label{algebre}
There are at least three different definitions of isoparametric   submanifolds. We first show that 
any BCH  $\Sigma^n$  is foliated by  the 
level sets of the mean curvature function $h$, which are  isoparametric codimension 2 submanifolds 
of $\mathbb{E}^{n+1}$, as defined by Terng  \cite{T}.
Second, these h-level sets can be also seen   as  hypersurfaces of   $\Sigma^n$. As such,  
they  are constant mean curvature  hypersurfaces of $\Sigma^n$, 
as  well as   locally parallel hypersurfaces (which are also $h$-level hypersurfaces of $\Sigma^n$).
Third, and from the dual perspective, the function $h$ is an isoparametric function on $\Sigma^n$, in the sense that  $|\nabla h|$ and $\Delta h$ are constant 
on the level sets of $h$ (see the beginning of the proof of Proposition \ref{evolution}).
\begin{prop}\label{para5}Suppose $\Sigma^n$ is proper.  The level sets of its  mean curvature function,
  $h,$  i.e., the leaves $U_t$,
are isoparametric within the meaning  of Terng (cf. \cite{T} and the introduction). 
Also, these level sets are parallel and are   of constant mean curvature as hypersurfaces of $\Sigma^n$.
 The mean curvature function, $h$, defined on $\Sigma^n$, is  also an isoparametric map, i.e.,
  $|\nabla h|$ and $\Delta h$ are constant on each leaf of $h$. 
 \end{prop}
\begin{proof}   First, we  show that the principal curvatures of the $h$-level set, $U_0$,  are constant. Indeed, by Lemma  \ref{paraODE},  the 
flow curves are congruent,   in particular,
by Proposition \ref{paraODE},  for any two points $s,s'\in U_{0}$  and $x_n=t$, in a small neighborhood of $0$, we have
\begin{equation} \label{paraODEbis}
 \tilde R(s,t):= \sum_{i=1}^{n-1}\frac{\lambda_{0i}(s)-\alpha' (t)\mu_{0i}(s)}{1-\lambda_{0i}(s)\alpha(t)-t\mu_{0i}(s)}=
 \tilde R(s',t):=
  \sum_{i=1}^{n-1}\frac{\lambda_{0i}(s')-\alpha' (t)\mu_{0i}(s')}{1-\lambda_{0i}(s')\alpha(t)-t\mu_{0i}(s')}.
\end{equation}
\begin{rem} We may as well  sum over the different normal curvature vectors, weighted by the multiplicity of the corresponding 
eigenspaces, and write equivalently
\begin{equation} \label{paraODEbis}
 \tilde R(s,t):= \sum_{i=1}^{r}\mathfrak{d}_i \frac{\lambda_{0i}(s)-\alpha' (t)\mu_{0i}(s)}{1-\lambda_{0i}(s)\alpha(t)-t\mu_{0i}(s)}=
 \tilde R(s',t):=
  \sum_{i=1}^{r}\mathfrak{d}_i \frac{\lambda_{0i}(s')-\alpha' (t)\mu_{0i}(s')}{1-\lambda_{0i}(s')\alpha(t)-t\mu_{0i}(s')},
\end{equation}
where the multiplicities are locally invariant.
\end{rem}
Let us consider  $c_i(s):=\frac{1}{\lambda_{0i}}(s)$  and  $b_i(s) =\frac{\mu_{0i}}{\lambda_{0i}}(s)$. 
First, we show the uniqueness of the coefficients $(b_i,c_i)$
in the  case  where $\alpha,\alpha'$ and $t$ are algebraically independent. Let us see 
$\alpha $ as a (convergent) power series of $\mathbb{C}[[t]]$. From its properties, $\alpha(0)=\alpha'(0)=0$
and $\alpha"(0) \neq 0$, since $h\neq 0$. Hence, 
$\alpha = l.t^2$ and  $\alpha' = k.t$, where $l$ and $k$ are invertible elements of   $\mathbb{C}[[t]]$, as  their zero order coefficient 
is nonzero (cf. equation \eqref{lambda2}).  Then, equation  \eqref{paraODEbis} becomes an equality 
  \begin{equation} \label{ODE5}
\sum_{i=1}^{n-1}\frac{-tkb_i(s)+1}{t^2l +tb_i(s) -c_i(s)} =
\sum_{i=1}^{n-1}\frac{-tkb_i(s')+1}{t^2l +tb_i(s') -c_i(s')} = \frac{f}{g}.
 \end{equation}
$\frac{f}{g}$ is a rational function of $t$ in the extension field $ K=\mathbb{C}(l,k)$, 
where $f$ and $g$ are polynomials in $t$ 
 with coefficients in the field  $K$ and $deg(f)<deg(g)$. Equation \eqref{ODE5} then provides two decompositions of the same $\frac{f}{g}$.
\begin{lem}\label{casTranscendant} If there are  two  points $s_1,s_2\in U_0$,  such that 
$\tilde R(s_1,t) = \tilde R(s_2,t)$, and there is no nonzero polynomial $P$ such that $P(t,k,l)=0$, then 
 $\left(b_i\left(s_1\right)),c_i\left(s_1\right)\right) = (b_{\sigma(i)}(s_2),c_{\sigma(i)}(s_2))$ for $i= 1,\ldots,n-1$ and some  permutation 
 $\sigma \in \mathfrak{S}_{n-1}$.
\end{lem}
\begin{proof}
Suppose that $\tilde R(s_1,t) = \tilde R(s_2,t)$, but with $\left(b_i\left(s_1\right)),c_i\left(s_1\right)\right) 
\neq  (b_{\sigma(i)}(s_2),c_{\sigma(i)}(s_2))$ for some $i\in \{1,\ldots, n-1\}$.\\
Consider any two different polynomials of the form $lt^2+bt-c$ and $lt^2+b't-c'$, 
  i.e., such that   $(b,c) \ne q (b',c')$, thus  
\begin{enumerate}
    \item 
    If $b=b'$, then 
  $c\neq c'$,  and the polynomials are coprime. 
  \item If $ b\neq b'$, then the Euclidean algorithm  yields,  at the second step,  the  remainder 
 $$r= -c'+\frac{c-c'}{(b-b')}\left(b' + \frac{(c-c')l}{(b-b')}  \right).$$
 \begin{enumerate}
 \item If $c=c'$,  then $r=-c'$, which is nonzero.
 \item  If $c\neq c'$, then the $l$ coefficient is nonzero; hence, $r\neq 0$. 
 \end{enumerate}
\end{enumerate}  
In all these cases, the remainder is a nonzero element of $K$; hence,  the denominators are coprime.
 In conclusion,   any  two polynomials of the form $lt^2+bt-c$ are  either equal (iff $(b,c)=(b',c')$),   or coprime (iff $(b,c)\neq (b',c')$).\\
 Now,  we consider the following  term 
  \begin{equation} 
 \frac{-tkb+1}{t^2l +tb -c}, 
 \end{equation}
 then,  either the fraction is irreducible, or
 $l= b^2k(-1+kc)$, and the fraction reduces to an irreducible fraction,
which cannot  happen since $l,k$  are algebraically  independent.
The denominator is  either irreducible in $K$  or a product of two  linear forms with 
distinct roots, because  the discriminant $b^2+4cl\neq 0$ since $c\neq 0$. 
  In the latter case, the corresponding fraction can   be explicitly and uniquely  split  as a sum of two 
fractions of the form:
 \begin{equation} \label{ODE7}
 \frac{-tkb+1}{t^2l +tb -c}  = \frac{\delta}{t-\rho_1} +  \frac{\epsilon}{t-\rho_2}
\qquad  \delta,\epsilon,\rho_1,\rho_2 \in K. \end{equation}
 In short, grouping the fractions  having 
 identical denominators,  we obtain a partial fraction decomposition of $\frac{f}{g}$, i.e., a sum  of irreducible fractions
 with coprime denominators  of order one or two. By the uniqueness of the partial fraction
 decomposition (cf. for example \cite{VDW}), the partial fraction
 decomposition of  $\tilde R(s_1,t)$ is 
  identical to  the partial fraction decomposition  of   $\tilde R(s_2,t)$, which 
  implies $\left(b_i\left(s_1\right)),c_i\left(s_1\right)\right) = (b_{\sigma(i)}(s_2),c_{\sigma(i)}(s_2))$, for $i= 1,\ldots,n-1$,
  for some permutation of the indices $\sigma$, which contradicts the hypothesis.
  \end{proof}
Hence,  we know that in the case of non-constant principal curvatures of $U_0$, there is a nonzero polynomial $P$, such that, $P(t,\alpha,\alpha')=0$. In fact, we have the following.
\begin{lem}\label{reductionAuCasAlgebrique} Suppose that some principal curvatures are non-constant in a neighborhood  $U_0$. Then, 
 there is a nonzero polynomial $P_1$ such that $P_1(t,\alpha)=0$, and 
$\alpha$  extends to a  non-constant  multi-valued holomorphic algebraic function on $\mathbb{C}.$ 
\end{lem}
\begin{proof}
To each point $q\in U_0$, denote by $\{\lambda_i(q)\}_{i=1,\ldots, r}$ the (ordered) spectrum  values  of respective multiplicities
$\{\mathfrak{d}_i\}_{i=1,\ldots, r}$.
For each given value  $\lambda_i(q),$ there may correspond possible different values of the spectrum, $\{\mu_{i_j}(q),\ j=1,\ldots, r_i\}$, associated to 
  the   normal curvature vectors $\{n_{i_j}(q)\}_{i_j=1,\ldots, r_i}=\{(\lambda_i(q),\mu_i(q))\}_{i=1,\ldots, r}$.
In that case, we  perturb the level set,
replacing  $U_0$ by the h-level set, $U_t$ for small $t>0$. Then, from Lemma \ref{courbureVar1}, each  $\lambda_i(q_t)$ takes different values  $\lambda_{i_j}(q_t)$. We can then suppose  that 
there are $r$ normal principal vectors, $n_i(q) := (\lambda_i(q),\mu_i(q))$, of multiplicities $\mathfrak{d}_i$, with 
 $r\leq n-1$ and $ \sum_1^r \mathfrak{d}_i =n-1$, and  where the $\lambda_i$'s (respectively  $\mu_i$'s) are pairewise distinct.\\
Suppose there is  a  curve  $\gamma : I=[0,s_1[ \longrightarrow U_0 $  such that $\gamma(0)=p$ and
$\lambda_i(s):=\lambda_i(\gamma(s))\neq \lambda_{0i}:= \lambda_i(p)$ for some indices $i$.
Also, from  \eqref{paraODEbis} let us write    $\tilde R(s,t)$  as
$$\tilde R(s,t)   := A_s(t,\alpha) - B_s(t,\alpha)\alpha',$$
where $s\in ]0,s_1[$, and let us denote their respective underlying rational functions as
$$\tilde R_s(x,y,z) =  A_s(x,y) - B_s(x,y)z.$$
Then,  from  \eqref{paraODEbis}
\begin{equation}\label{egalite}
 A_{0}(t,\alpha) -B_{0}(t,\alpha) \alpha' = A_{s}(t,\alpha) -B_{s}(t,\alpha) \alpha' . 
 \end{equation}
Thus, let 
$$R_s(x,y):= \frac{A_s(x,y)-A_0(x,y)}{B_s(x,y)-B_0(x,y)}$$
for $s>0$.\\
 $R_s(x,y)$ is well-defined, i.e.,   $A_s-A_0$ and $B_s-B_0$ are nonzero  in $\mathbb{C}(x,y)$
for small positive $s$. Indeed, 
$$ A_s(x,y)-A_0(x,y) =\sum_{i=1}^r \left(\mathfrak{d}_i\frac{\lambda_i(s)}{1-\lambda_i(s)y-\mu_i(s)x}-
\mathfrak{d}_i\frac{\lambda_{0i}}{1-\lambda_{0i}y-\mu_{0i}x}\right),$$
in particular, 
$$ A_s(0,y)-A_0(0,y) =\sum_{i=1}^r \left(\mathfrak{d}_i\frac{\lambda_i(s)}{1-\lambda_i(s)y}-
\mathfrak{d}_i\frac{\lambda_{0i}}{1-\lambda_{0i}y}\right).$$
Since the $\lambda_{0i}$ are mutually  distinct, then  so are the  $\lambda_{i}(s)$ for small $s$.  Thus,  by the uniqueness of the partial fraction decomposition, 
and since at least one term of the RHS sum is nonzero, the numerator $ A_s(x,y)-A_0(x,y)$ of $R_s(x,y)$ is nonzero.
If $ B_s(x,y)=B_0(x,y)$,  
then, from equation \eqref{egalite}, $A_s(t,\alpha)-A_0(t,\alpha) =0$, which yields  a non-trivial polynomial $P$  such that $P(t,\alpha)=0$. 
If, on the other hand, the denominator
 $ B_s(x,y)-B_0(x,y)$ is a nonzero  rational function, then,  $R_s(x,y)$  is a  nonzero rational function. Hence, for $s'$ close to $s$, 
the numerator $N(x,y)$ of  $R_s(x,y)-R_{s'}(x,y)$ when reduced to a common denominator. $N$ is not trivial. 
Indeed,  in terms of the curvature radii in $R_s$,   $\rho_i := \frac{1}{\lambda_i}$ and ratios  $\tau_i:= \frac{\mu_i}{\lambda_i}$
(respectively  $\rho'_i := \frac{1}{\lambda'_i}$ and ratios  $\tau'_i:= \frac{\mu'_i}{\lambda'_i}$ in $R_{s'}$), we have 
\begin{align*}
    N(0,y)&=\left[\sum_{i=1}^r \left( \frac{\mathfrak{d}_i}{\rho_i -y}  -  \frac{\mathfrak{d}_i}{\rho_{0i} -y} \right)\right]
\left[\sum_{i=1}^r \left( \frac{\mathfrak{d}_i\tau'_i}{\rho'_i -y}  -  \frac{\mathfrak{d}_i\tau_{0i}}{\rho_{0i} -y} \right)\right]
\\
&-
\left[\sum_{i=1}^r \left( \frac{\mathfrak{d}_i}{\rho'_i -y}  -  \frac{\mathfrak{d}_i}{\rho_{0i} -y} \right)\right]
\left[\sum_{i=1}^r \left( \frac{\mathfrak{d}_i\tau_i}{\rho_i -y}  -  \frac{\mathfrak{d}_i\tau_{0i}}{\rho_{0i} -y} \right)\right].
\end{align*}
If $N(0,y)$ is zero, then, as the $\rho_i, \rho_{0i},\rho'_i,\rho'_{0i}$ are mutually distinct, $\tau_i=\tau'_j =\tau_{0k} $  for any $i,j,k,$ i.e., 
all the $\tau$s are equal to some constant $c$. Then,  $R_s(x,y)=c$ which implies $R_s(t,\alpha)=\alpha' = c$, i.e.,   the flow curve is a line 
which is ruled out.
 In conclusion, there is  a non-trivial polynomial, $N$, such that $N(t,\alpha)=0$;
$\alpha$ is then an algebraic function   (see for example \cite{FK}), which  extends to a $d$-valued holomorphic algebraic function still denoted by $\alpha$
on $\mathbb{C}$  and on  $\mathbb{S}^2:=\mathbb{C}\cup \{\infty\}$, where $d\geq 1$ is the degree of the polynomial of $P(t,\alpha)$ with respect to  $\alpha$
 defining a ramified cover of  $\mathbb{C}$ or   $\mathbb{S}^2$ of degree $d$. 
  \end{proof}
In this final paragraph, we show that $\alpha$ cannot be algebraic  by  investigating the behavior, at infinite points,  of the 
 solution on $\mathbb{C}\cup\{\infty\}$ of  the ODE  \eqref{lambda100}. Together  with the previous lemma, this will lead to 
a contradiction and will show that the principal curvatures along the level sets $U_t$ must be  constant.
  \begin{lem}\label{casAlgebrique}  There exists no  nonzero  algebraic function 
 solution of the ODE and initial conditions  \eqref{lambda100}
 with a pole at some finite point if $r_0 :=\sum_{ i=1}^{n-1}\frac{\mu_{0i}}{\lambda_{0i}}\neq 0.$
 Similarly, there is no pole at infinity if $s_0 :=\sum_ {i=1}^{n-1}\frac{\lambda_{0i}}{\mu_{0i}}\neq 0.$
\end{lem}
\begin{proof}
Notice first that we can choose  $\mu_i$ to be non zero by a small perturbation of the level set
(see Lemma \ref{courbureVar1}).
 Suppose that $\alpha$ is a (possibly multi-valued) algebraic 
meromorphic  map on $\mathbb{C}$ which satisfies the   ODE   \eqref{lambda100}.
Consider   the behavior of $\alpha$ at a point $t_0\in \mathbb{C}$, where $\alpha$ is infinite. Then,
$\alpha \sim  a u^\nu$ and  $\nu <0,$ when  $t\mapsto t_0$, i.e.,  when $u:=t-t_0 \mapsto 0$. 
At $u=0$, the LHS of the ODE \eqref{lambda100} behaves like $\frac{(\nu-1)}{\nu a u^\nu}$;   hence, it is finite, and even zero. 
However,   the Puiseux  expansion of the  RHS   yields
$$ r_0\nu u^{-1}   +O( u^{-\nu}).$$
 Thus,  a  meromorphic solution to \eqref{lambda100}  which has a pole on $\mathbb{C}$  must satisfy $r_0=0$. 
 \end{proof}
 To end  the proof of Proposition \ref{para5}, let us show that, 
by  performing  a small perturbation along the flow curve,  $r_0$ can be chosen nonzero. This shows that there is 
no non-trivial algebraic solution to equation \eqref{lambda100} with the given initial data.\\
For $\epsilon$ sufficiently small, consider the flow curve  starting at the point $p_1:=(t_0,\alpha(t_0))$, 
for a small $|t_0| <\epsilon$,  and expressed in the rotated  frame 
$\{e_{t_0},N_{t_0}\}$. The flow curve is then given by the position vector $m_{p_1} (u) = ue_{t_0} +\alpha_1(u)N_{t_0}$
with the initial conditions $\alpha_1(0) = \alpha_1'(0) = 0$, where $\alpha_1$ is the solution of the
ODE \eqref{lambda100} in which   the set of coefficients  $\{\lambda_{0i},\mu_{0i}\}_{i =1,\ldots, n-1}$ are  now replaced
by the set of coefficients  $\{\lambda_{t_0i},\mu_{t_0i}\}_{i =1,\ldots, n-1}$, and with the same initial conditions as before.\\
Also, applying the equations given in Lemma \ref{courbureVar1},
 we can measure the variation of the coefficient $r_t$ along the flow curve. We have
 \begin{equation}
r_t =\sum_{\lambda_{0i}\neq 0} \frac{\mu_{ti}}{\lambda_{ti}} = 
\sum_{\lambda_{0i}\neq 0} \frac{\mu_{0i}+\alpha'(t)\lambda_{0i}}{\lambda_{0i}-\alpha'(t)\mu_{0i}}.
\end{equation}
As $\alpha'(t)=a_1t+o(t)$, where $ a_1\neq 0$ since $\alpha''(0)\neq 0$, 
the expansion of $r_t$  for small $t$ yields
\begin{equation}
r_t =\sum_{\lambda_{0i}\neq 0}\frac{\mu_{0i}}{\lambda_{0i}} +
 \alpha'(t)\sum_{\lambda_{0i}\neq 0}\left( \frac{\lambda_{0i}^2+\mu_{0i}^2}{\lambda_{0i}^2}    \right) +o(t),
\end{equation}
or, in our case, since $r_0=0$,
\begin{equation}
r_t = 
a_1t\sum_{\lambda_{0i}\neq 0}\left( \frac{\lambda_{0i}^2+\mu_{0i}^2}{\lambda_{0i}^2}    \right)+o(t),
\end{equation}
which is nonzero for a sufficiently small positive $t_0$. Applying  Lemma \ref{casAlgebrique},
we deduce that, by a small perturbation of the starting level set, 
the solution $\alpha$ is holomorphic on $\mathbb{C}$.\\
Now, we consider the behavior of  a solution  at infinity of  the same    ODE  with the same  initial conditions   \eqref{lambda100}.
$\alpha$ has necessarily a pole at infinity; otherwise,  $\alpha$ would be bounded and hence constant,  which cannot be
the case, thus $\alpha \sim  a x^\nu, \nu>0$ as  $x \mapsto \infty$.
Comparing  the LHS and the RHS of the ODE \eqref{lambda100} at infinity,
 we deduce that the only possible behavior is a pole of order one for each branch, where the value of $\alpha$  is infinite.
There, using the  change of variable $u= \frac{1}{x}$, we have  $\alpha \sim a u$ as $u\mapsto 0$. The LHS 
behaves like $\frac{2u^3a}{1+a^2}+O(u^3) $  as  the RHS  
 asymptotic expansion at infinity  yields  
$$ -s_0 u  + O(u^2):= \sum_i \frac{\lambda_{0i}}{\mu_{0i}} u + O(u^2). $$
 Then, we  deduce   that there is no non-trivial algebraic $\alpha$ satisfying the ODE with the initial conditions  \eqref{lambda100}  when 
$r_0\neq 0$ and $s_0\neq 0$.\end{proof}
We  finally conclude  from Lemmas \ref{casTranscendant}, \ref{reductionAuCasAlgebrique} and \ref{casAlgebrique}    that 
the principal curvatures of the $h$-level set $U_0$,
 or possibly for the $h$-level sets $U_\epsilon$ for some small  perturbation ($\epsilon_0>\epsilon>0$)
  must be  constant. If $s_0=0$, then, replacing $(\lambda_i,\mu_i)$ by $(-\mu_i,\lambda_i)$ in the computation of $r_t$ 
  also gives $s_t\neq 0$. Moreover,  the normal bundles are  flat by Lemma \ref{para1}, 
so  $ U_{t}$  is isoparametric (cf. the introduction).   Notice that, once we know that $U_0$ is isoparametric,
 then  $U_t$ is  isoparametric for any sufficiently small $t$, since they are parallel to $U_0$: $U_t = U_0 +\zeta_{t}$, where  $\zeta_{t}$ is a normal parallel vector field (see   Lemma \ref{para4}),  as long as 
no point reaches a focal point on a focal line of $U_0$.
 \section{On    curvatures   of  the $h$- level sets $U_t$   of  BCHs }\label{curvatureLevelSets}
We compute the variation of the curvatures of 
 the isoparametric level sets $U_t$ along the flow curve (see Lemma \ref{para4}).
\begin{prop}\label{evolution}
 Let $\lambda_i, i=1,\ldots n-1$ 
 (respectively  $\mu_i, i=1,\ldots n-1$) 
 be the principal curvatures  of the shape operator $A_N$ (respectively    of the shape operator $A_{e_n}$)
  of an $h$-level surface $U_t$ of a proper BCH $\Sigma^n$.\\
 Then, the  Riccati equation yields 
\begin{equation}\label{Riccati10}
 \mu_{i,n} = \mu_i^2 + \lambda_i\lambda_n, \quad  i\in \{1,\cdots,n-1\},
\end{equation}
  and the  Codazzi equation yields
\begin{equation}\label{Codazzi10}
 \lambda_{i,n}= \mu_{i} (\lambda_i -\lambda_n),  \quad  i\in \{1,\cdots,n-1\}.
\end{equation}
\end{prop}
\begin{proof} Let us first prove the first  order ODE \eqref{Riccati10}.
 The level sets of  $h$  are, equivalently,  the level sets of the principal curvature $\lambda_n$, which is
the principal curvature along a flow  curve.
The length  element along the  curve, $ds = \sqrt{1+\alpha'(t)^2} dt$, depends only on the parameter  $t$. Hence, the level sets of $h$ 
are equidistant, and $|\nabla h|$ is constant on each $h$-leaf. \\
The variation of the principal curvatures will be given by the Riccati and  the Codazzi equations.
The Riccati equation of the second fundamental form $A_{e_n}$ of $U$, as a hypersurface of $\Sigma^n$, is
$$A_{e_n}' = A_{e_n}^2 +R_{e_n},$$ 
where $R_{e_n}(e_i) = R(e_i,e_n)e_n, i =1,\ldots n-1,$ and $R_{e_n}$  is the $e_n$-directional  curvature  operator (see for example \cite{K}). \\
 Gauss equation of  $\Sigma^n$, as  a hypersurface of $\mathbb{E}^{n+1}$, gives: 
$$R(X,Y)Z =(A_NX\wedge A_NY)Z:=\langle A_NY,Z\rangle A_NX -\langle A_NX,Z\rangle A_NY,$$
 (see for example \cite{DT}). Hence, 
$$R_{e_n}(e_i)= R(e_i,e_n)e_n = \lambda_n\lambda_i  e_i. $$
Thus, we get the  evolution of the principal curvature in the $e_n$-direction
   $$\mu_{i,n} = \mu_i^2 + \lambda_i\lambda_n, $$
and so we get also the evolution of   $M_1 := \sum_1^{n-1} \mu_i $ 
$$M_{1,n}=  M_2 -3\lambda_n^2 :=\sum_1^{n-1}\mu_i^2 -3\lambda_n^2.$$
The second ODE \eqref{Codazzi10} follows   from  the  Codazzi equation \eqref{total}
 $$ \mu_{i} e_i = A_{e_n}{e_i} = - \nabla_{e_i} e_n = \frac{\nabla_{e_n} \lambda_i}{\lambda_i -\lambda_n}e_i, 
 $$ 
 so
 $$  \mu_{i} =   \frac{ \lambda_{i,n}}{\lambda_i -\lambda_n}.$$ \end{proof}
  \section{Proof of   Theorem \ref{letheoreme1}  }\label{constructionBCH}
  From the  preceding lemmas  \ref{para1}  to  \ref{paraODE}  and Proposition \ref{para5}, a   proper BCH  $\Sigma^n $ in $ \mathbb{E}^{n+1}$ is  foliated by 
the level sets of the mean curvature function  $h$ of $\Sigma^n$, which  are isoparametric codimension
 two submanifolds  of $\mathbb{E}^{n+1}$ in 
a neighbourhood  of a regular point of $h$.\\
The symmetries of the leaves extend to symmetries on $\Sigma^n$.
Indeed,  Let $G$ be the subgroup  (possibly trivial) of the ambient isometries of $\mathbb{E}^{n+1} $ that preserve $U_0$.
Let $g\in G$  and $p\in \Sigma^n$, then $p $ belongs to an integral curve of $\nabla h$
that cuts $U_0$ at a point $p_0 \in U_0$. Let $q_0 :=g(p_0)\in U_0$, the normal plane,  $N_{p_0}U_0$, is sent by the isometry $g$ 
to $g(N_{p_0}U_0)= N_{q_0}U_0$, and  the normalized mean curvature at $p_0$, $N(p_0)$, 
is sent to   $ N(q_0)$. Hence,  we also have  $g(e_n(p_0)) = e_n( q_0)$. By  the uniqueness of the ODE \eqref{lambda100}, with given
initial conditions and by Lemma \ref{para5},
 the  integral curve $N_{p_0}U_0 \cap \Sigma^n$ is sent to
$g(N_{p_0}U_0 \cap \Sigma^n) = N_{q_0}U_0 \cap \Sigma^n.$ 
Consequently, $g(\Sigma^n)=\Sigma^n$.\\
 Conversely,  let $U_0$  be a given nonempty  open set of an    isoparametric submanifold $ U_0\subset \mathbb{E}^{n+1}$
 of codimension 2 with nonzero mean curvature vector field. Then,
 there exists  a nonempty  open subdomain $V_0\subset U_0$   and a  
proper BCH   diffeomorphic to $V_0\times I$, ($I$ open interval of $\mathbb{R}$), 
 obtained by a local normal flow  of submanifolds parallel to $U_0$ in the normal bundle of $U_0$ that preserves the symmetry group of $U_0$.

 \subsection{Existence of proper BCHs}
Let us construct a  BCH once  an isoparametric submanifold of codimension 2 is given.
The  second order ODE \eqref{lambda100} is of the form 
$ \alpha'' =  R(x_n,\alpha,\alpha')$,  where $R(x,y,z)$ is a  rational function  of its variables, 
 which is analytic  on  $\mathbb{R}^2 \setminus L$, where $L$ is the union of focal lines
 $ \bigcup _{i=1}^{n-1} l_{0i} = \bigcup_{i=1}^{n-1} \{ v=(x,y) \in \mathbb{R}^2 : 1-\langle n_{0i},v\rangle = 0 \}$, and where 
\begin{equation}\label{R}
R(x,y,z) = -\frac{(1+z^2)}{3}\sum_{i=1}^{n-1}\frac{\lambda_{0i}+z \mu_{0i}}{1-\lambda_{0i}y-x\mu_{0i}}.
\end{equation}
Hence, $\alpha$ is  locally the unique   solution with, by construction,
the initial conditions $\alpha(0)=0,$  $\alpha'(0)=0$ and $\alpha"(0)\neq 0$, as  the flow curve starting at the origin $O$, is tangent to $e_n$ at $O$
and $h\neq 0$.\\
 Given an isoparametric immersion, $ Y: D \longrightarrow U_0\subset \mathbb{E}^{n+1}$, of codimension 2. 
 Let $\{e_{0n}, N_0\}$ be the  parallel frame field of its normal bundle, and let $H_U$  denotes its mean curvature vector field.
 We consider the immersion 
 \begin{equation}
 \begin{array}{ll}
 X: D\times I& \longrightarrow \mathbb{E}^{n+1}\\
 (p,t)&\mapsto Y(p) + \zeta(p,t),
 \end{array}
 \end{equation} 
 where 
 $\zeta(p,s) $ is the solution of  ODE \eqref{lambda120} (which is a different way of writing the ODE   \eqref{lambda100}). Let us check that $X$ is a BCH.\\
 It suffices to verify that $e_n$,  the unit normal vector field $\frac{\nabla h}{|\nabla h|}$ on an open subdomain, where  $h\ {\rm and} \nabla h \neq 0$,
 is a principal curvature direction of principal curvature equal to $-\frac{\Lambda_1}{3}$.\\
 Indeed, it is clear that  planar vectors $\zeta_{,ss} \perp \zeta_{,s}$,  so  from  \eqref{lambda120}, we have 
 $$\lambda_n   = \left\langle ( D_{\zeta_{,s} } \zeta_{,s}),  J(\zeta_{,s})   \right\rangle= 
  \left\langle \zeta_{,ss} ,  J(\zeta_{,s})   \right\rangle= -\left\langle \frac{(n-1)}{3}H_U, J(\zeta_{,s})\right\rangle= -\frac{\Lambda_1}{3}.$$
  Notice that, the same reasoning applies to the ODE \eqref{lambda100}, where the coefficient $-1/3$ is replaced by $-1$,
  so that $\lambda_n = -\Lambda_1$, thus, generating  minimal hypersurfaces instead of BCHs.
  \begin{cor}
  Let $Y:U \mapsto \mathbb{E}^{n+1}$ be a codimension 2 open isoparametric submanifold of $\mathbb{E}^{n+1}$. 
  Then, there is a nonempty  open subset $V\subset U$, a nonempty interval $I:=[0,\epsilon[$ and a minimal immersion $X: V\times I \mapsto \mathbb{E}^{n+1}$ such that 
  $X(V\times \{0\} )=Y(V)$.
  \end{cor} 
     \subsection{Examples of BCHs}\label{examplesOfBch}
       \begin{prop} \label{SigmaU} Let  $U_0$ be  an  isoparametric  hypersurface of
        $\mathbb{E}^n \subset \mathbb{E}^{n+1}$. There exists a proper   BCH   
containing $U_0$   and whose  $n$-plane containing $U_0$ is a plane of  symmetry.
\end{prop}
 \begin{proof} We refer to the notations of Subsection \ref{ODE0}.
We consider the parametrization \eqref{surRev2}, where $ e_n$ is  a unit normal vector field to the n-plane containing $U_0$. 
The principal curvatures $\mu_i$, i.e., the eigenvalues of the shape operator $A_{ e_n} $, are zero on $U_0$. Hence, 
the second order ODE  \eqref{lambda100}, which characterizes the BCH  becomes
\begin{equation}\label{lambda101}
\frac{\alpha''}{1+\alpha'^2}=-\frac{1}{3}\sum_{i=1}^{n-1}\frac{\lambda_{0i}}{1-\lambda_{0i}\alpha}= R(\alpha).
\end{equation}
Multiplying by $\alpha'$ allows us to obtain a first quadrature 
 $$1+\alpha'^2 = C_2\prod_{i=1}^{n-1} (1-\lambda_{0i}\alpha)^{2/3}.$$
 The initial conditions,
$\alpha(0) = \alpha'(0)=0$, imply $C_2=1$. Choosing the orientation of the normal vector field $N$, such that $h<0$, we then  obtain $ \prod_{i=1}^{n-1}  (1 -\lambda_{0i}\alpha)^{2/3} -1$
is nonnegative for $\alpha$ positive and
\begin{equation}\label{alphaPrime}
 \alpha' = \pm\sqrt{ \prod_{i=1}^{n-1}  (1 -\lambda_{0i}\alpha)^{2/3} -1},
 \end{equation}
  so 
   $$ \frac{d\alpha}{  \sqrt{ \prod_{i=1}^{n-1}  (1 -\lambda_{0i}\alpha)^{2/3} -1}} = \pm dt.$$
 A second integration yields an expression of the inverse  of the function $t = f(\alpha)$, and the value of $\lambda_n$
 from equation \eqref{lambda2} is
 \begin{equation}\label{integrale1}
t = C_3 \pm \int_0^\alpha  \frac{dx}{  \sqrt{ \prod_{i=1}^{n-1}  (1 -\lambda_{0i}x)^{2/3} -1}} \ {\rm and} \
\lambda_n = 
\frac{1}{3}\left(\sum_{i=1}^{n-1} \frac{\lambda_{0i}}{1-\lambda_{0i}\alpha}\right) \prod_{i=1}^{n-1}(1-\alpha\lambda_{0i})^{-1/3}.
 \end{equation}
 
The initial condition
$\alpha(0)  =0$, implies $C_3=0.$
  Notice that, the solution, $x$, of equation \eqref{integrale1}  is an   integral, which is equivalent, near $0$, to $\int_0^\cdot \frac{dx}{  \sqrt{ x}}$. Hence, $x$ is  integrable if 
$\sum_1^{n-1} \lambda_i \neq 0$ (if
$\sum_1^{n-1} \lambda_i= 0$,  then  $\lambda_n=0$ and $\Sigma^n$ would be  minimal).  \end{proof} 
 \begin{figure}[h!]\label{fig1}
\begin{center}
\includegraphics[scale=0.45]{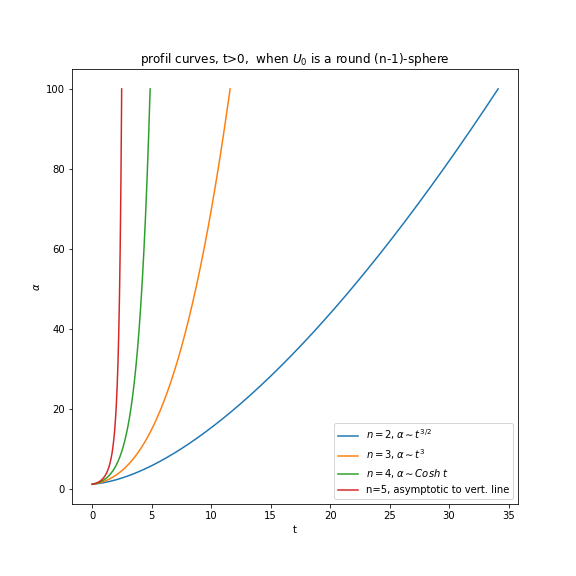} \\
\begin{center}
{\small Fig 2.  Profile curves of rotational BCHs }
\end{center}
\end{center}
\end{figure} 
Let   $U_0$ be  an $(n-1)$-sphere, that is, an orbit of the  
 subgroup $O(n)$ of isometries of the ambient Euclidean space $\mathbb{E}^{n+1}$. Then $U_0$ is isoparametric,  and  its
  principal   curvatures $\lambda_{0i}$'s are all equal to $\lambda_0$. Extending $U_0$ via $\alpha$,  we obtain
 BCHs that are  rotationally   invariant  ( cf. details in \cite{MOR} or \cite{N}, and see Fig. 2 for the profile curves).
 Notice also that  the curve $\alpha$  satisfies an algebraic, first order differential equation   with principal curvature $\lambda_n$, and  whose 
 inverse function   is an elliptic integral:
\begin{equation}\label{elliptique}
(1+\alpha'^2)^3 =(1-\lambda_0 \alpha)^{2(n-1)}, \quad  \lambda_n = 
\frac{(n-1)\lambda_0}{3(1-\lambda_{0}\alpha)^{\frac{n+2}{3}}}\ {\rm and}  \ \alpha^{-1}(t) = \int_0^t \frac{dx}{  \sqrt{    (1 -\lambda_{0}x)^{\frac{2(n-1)}{3}} -1}}.
\end{equation}

\section{Appendix:  Codazzi's equations   and  principal  curvature directions  }\label{cristoffel}
One deduces directly from the Codazzi equations two key properties  related to the
principal curvatures of a hypersurface $\Sigma^n\subset \mathbb{E}^{n+1}$.
 Recall that  the shape operator in the normal direction $N$, $A_N$,  decomposes $T\Sigma^n$ into distributions
 $T_i,$ where   $T_i = ker(A_N(p)-\lambda_i Id)$ and  $\lambda_i $
  is the corresponding eigenvalue (or principal curvature)   of multiplicity $rk(T_i)=m_i.$   We will select  an open
   subdomain  where  the multiplicities $m_i,\   i=1,\ldots, s,$ of the corresponding
  eigenvalues  are constant.  Then, we can split   the tangent bundle $T\Sigma^n$   into  a family of integrable
  distributions  $T_i\subset T\Sigma^n $.
Let
$X,Y \in T\Sigma^n $ two tangent vector fields, and let $A=A_N$ be the shape operator of $\Sigma^n$, then the Codazzi equation is
\begin{equation}\label{codazzi}
 \nabla_X(A)Y  =   \nabla_Y(A)X,
 \end{equation}
  where
 $\nabla_X(A)Y  = \nabla_X(AY) - A(\nabla_X Y).$
\begin{enumerate}
\item   Let  $T_\lambda$ be  the  distribution  which corresponds to the  eigenvalue $\lambda$ of the shape operator $A$  of $\Sigma^n\subset \mathbb{E}^{n+1}$, with locally constant rank larger than 2. Then, it is integrable
  and it defines a curvature surface $\mathcal{E}_\lambda$ on which $\lambda$ is constant.\\
  Indeed, let
$X,Y \in T_\lambda$ be linearly independent  vector fields in the given distribution $T_\lambda$. Then, by  the  Codazzi equation:
\begin{equation}\label{coda}
\nabla_X(A)Y = \nabla_Y(A)X,
\end{equation}
so
$$(X\lambda)Y -(Y\lambda)X = (A-\lambda Id)[X,Y].$$
As the LHS belongs to $ker(A-\lambda Id) $ and the RHS belongs to
$Im(A-\lambda Id)$, we obtain $X(\lambda)=Y(\lambda) = 0$ and $[X,Y] \in T_\lambda$.
 Thus, from Frobenius, the distribution associated to the principal eigenvalue $\lambda$ is indeed integrable into a curvature surface 
 $\mathcal{E}_\lambda$  onto which   $\lambda$ is constant.
\item   $\mathcal{E}_\lambda$  is totally umbilical, i.e., the second fundamental form of $\mathcal{E}_\lambda$  is equal to
$$II(X,Y) = \langle X, Y\rangle H,\quad  \forall X,Y \in TS(\lambda) ,$$
where $ H$ is the mean curvature vector field of $\mathcal{E}_\lambda$.\\
Indeed,
let $X$ be a  vector field in $T\mathcal{E}_\lambda$, and let $\eta\in T\Sigma^n$ be  a vector field normal to $\mathcal{E}_\lambda$
such that $A\eta =\mu\eta$ and $\mu \neq \lambda.$ Then,
 $$\nabla_X(A)\eta = \nabla_\eta(A)X.$$
 Expanding
 $$(\mu Id -A)(\nabla_X\eta) =
 (\eta\cdot\lambda)X- (A-\lambda Id) (\nabla_\eta X) - (X\cdot\mu)\eta.$$
 The projection  operator onto $T\mathcal{E}_\lambda$ commutes with $A$, thus
 $$(\mu Id -A)(\nabla_X\eta)^{T\mathcal{E}_\lambda} =
 (\eta\cdot\lambda)X- (A-\lambda Id) (\nabla_\eta X)^{T\mathcal{E}_\lambda} = (\eta\cdot\lambda)X.$$
Hence,
\begin{equation}\label{total}
(\nabla_X\eta)^{T\mathcal{E}_\lambda} =\frac{(\eta\cdot \lambda) X}{\mu - \lambda}.
\end{equation}
Let   $\{e_k\}_{k=1,\ldots ,l}$ be  an orthonormal  basis of $T\mathcal{E}_\lambda^\perp\cap T\Sigma^n$ in the principal curvature directions,
so that $(\nabla_{e_k} N )^{T\mathcal{E}_\lambda}= -\lambda_k e_k,$
where $N$ is a unit normal  vector  field to  $T\Sigma^n$, then the  second fundamental form of
$\mathcal{E}_\lambda \subset \mathbb{E}^{n+1},$
according to equation \eqref{total}, is  given by
$$II_\lambda(X,Y) = \sum_{k=1}^n - \langle \nabla_{X} e_k,Y\rangle e_k -  \langle \nabla_{X}N,Y\rangle N
=     \sum_{k=1}^n \frac{e_k\cdot\lambda}{\lambda-\lambda_k }\langle X,Y\rangle e_k  - \lambda  \langle X,Y\rangle  N.      $$
So
\begin{equation}    II_\lambda(X,Y)
 = \langle X,Y\rangle H_\lambda \  \text{and} \  H_\lambda = \sum_{i=1}^l  \frac{(e_i \cdot \lambda)}{\lambda_i - \lambda} e_i -
\lambda N,
\end{equation}
 where $H_\lambda$ is the mean vector field of $\mathcal{E}_\lambda\subset \mathbb{E}^{n+1}$.
$\mathcal{E}_\lambda$ is thus totally umbilical (see next Appendix \ref{spheres}).
\end{enumerate}
Let
$\Sigma^n$ be parametrized by  an  immersion
\begin{equation}
\begin{array}{cc}
X:&\Omega \subset \mathbb{R}^n\longrightarrow \Sigma^n \subset \mathbb{R}^{n+1} \\
&(x_1,\ldots,x_n) \mapsto X(x_1,\ldots,x_n),
\end{array}
\end{equation}
where $\Omega$ is an open domain, and
$(x_1,\ldots, x_n)$ form a system of orthogonal coordinates. Then, the metric $g$ in these coordinates is diagonalized as follows
\begin{equation}
g_{ij} =\delta_{ij}E_i= \delta_{ij}v_i^2 .
\end{equation}
The Christoffel Symbols of the metric $g$ are
$$\Gamma^i_{jk} = \frac{g^{ii}}{2}\left( g_{ki,j} +  g_{ij,k} -g_{jk,i}  \right) = \frac{1}{2E_i}\left( \delta_{ki}E_{k,j}  + \delta_{ij}E_{i,k} -\delta_{jk}E_{j,i}  \right). $$
If $i\neq j\neq k$,  then $\Gamma^i_{jk} = 0$.\\
Hence, the only possible  nonzero Christoffel symbols up to permutation on lower indices are
\begin{itemize}
   \item If $i = j \neq k$,  then
   $
   \Gamma^i_{ik} =  \frac{E_{i,k}}{2E_i} =(\log v_i)_{,k}.
   $
   \item If $j= k \neq i$, then $\Gamma^i_{jj} =  -\frac{E_{j,i}}{2E_i}=-(\log v_j)_{,i}\frac{v_j^2}{v_i^2}. $
\item If $j= k = i$, then $\Gamma^i_{ii} =  \frac{E_{i,i}}{2E_i}=(\log v_i)_{,i}$.
\end{itemize}
In the frame of this coordinate system, where
$e_i = \partial_i, \ i=1,\ldots, n$, the connection on $T\Sigma^n$ is given by the Cristoffel symbols
\begin{equation}\label{nabla}
\nabla_{e_i}e_k = \Gamma_{ik}^\alpha e_\alpha =\Gamma_{ik}^i e_i +\Gamma_{ik}^k e_k.
\end{equation}

In the holonomic case, when the curvature lines form a coordinate system,
the shape operator $A$
is diagonalized and related to the second fundamental form $II$  as follows
$$ A_i^j = \delta_i^j \lambda_i\ {\rm and}\ II_{ij} = A^\alpha_i g_{\alpha j}=  \delta_{ij}\mu_i,$$
where $\{\lambda_i\}_{i=1}^n$ are  the  principal curvatures
and  $\mu_i = \lambda_iv_i^2$.
The Codazzi equations   are equivalent to the following  system of
$n(n-1)$ equations:
\begin{equation}\label{codazzi}
\lambda_{i,j}-(\lambda_j -\lambda_i) \Gamma^i_{ij}=0,\quad \forall i\neq j\quad  i,j \in \{1,\ldots, n\}.
\end{equation}

 \section{Appendix: totally umbilical submanifolds of $\mathbb{E}^{n+1}$  }\label{spheres}
Let us recall that 
a submanifold $\Sigma^n$  is umbilical at $p\in \Sigma^n$ if there exist a unit normal  field $N(p)\in N_p\Sigma^n$ and curvature 
$\lambda\in\mathbb{R}$  such that
$$II(p)(X,Y) = \langle X,Y\rangle \lambda(p) N(p),\qquad  \forall X,Y\in T_p\Sigma^n. $$
By definition, $\Sigma^n$ is   totally umbilical if   it is  umbilical at each point of $\Sigma^n$.
 \begin{lem}\label{lemma6}
 A totally umbilical hypersurface  $\Sigma^n \subset \mathbb{E}^{n+1}$ is a piece of a hypersphere or a hyperplane.
 \end{lem}
 \begin{proof}
 Let $N$ be a unit normal vector field to $\Sigma^n$.
 Let $\gamma(s)$ be  a curve in  $\Sigma^n$ and let $T$ be the unit  vector field tangent  to $\gamma(s)$. Then, since $\frac{DN}{ds} $  and  
  $ \nabla_T N$ are both   orthogonal to $N$, we obtain
 
 \begin{equation}\label{normal1}
D_T N=  \frac{DN}{ds} = \nabla_T N. \end{equation}

And as  any direction in $T\Sigma^n$  is a principal curvature direction, we have 
 \begin{equation}\label{normal2}
 \nabla_T N= -\lambda T.
\end{equation}
From Appendix \ref{cristoffel}.1,   $\lambda$ is constant. Integrating  Equation \eqref{normal2}, we obtain
 $$N(s)-N(0) =-\lambda \int  \frac{DX}{ds}  ds = -\lambda\left( X(s)-X(0)\right).$$
 Hence, if $\lambda \neq 0$, we get 
 $$X(s) = X(0) - \frac{1}{\lambda}(N(s)-N(0) ):= C(0) -\frac{1}{\lambda}N(s). $$
 Thus  $\Sigma^n \subset \mathbb{S}^{n}(C(0),\frac{1}{\lambda}).$
 If $\lambda =0$, $\Sigma^n $ is a piece of a  hyperplane. More generally
 \end{proof}
 
 \begin{lem}\label{7}
 A totally umbilical submanifold   $\Sigma^k \subset \mathbb{E}^{n+1}$ with $k>1$  and with the extra condition  that the 
 mean curvature is parallel,  is a piece of a round sphere.
 \end{lem}
 \begin{proof}
 Let $(T_1,\ldots, T_k) $ be an orthonormal basis at some point extended to orthonormal frames  around some point.
 The unit normal vector field  $N =\frac{H}{h}$,  seen as a section of the normal bundle of $ \Sigma^k $, is  parallel, hence,
 for any $i = 1,\ldots, k$,  $(D_{T_i}N)^\perp=0$. 
 Thus, we deduce equation \eqref{normal2}  and, as in Lemma \ref{lemma6}, we conclude  that  
  $\Sigma^k\subset \mathbb{S}^{n}(C(0),\frac{1}{\lambda})$.\\
  On the other hand, consider the $(k+1)$-vector field $ W:=T_1\wedge\cdots\wedge T_k\wedge N$;
 then    
 $$D_{T_i} W = \sum_{\alpha =1}^k T_1\wedge\cdots\wedge  D_{T_i}T_\alpha \wedge\cdots \wedge T_k\wedge N+
  T_1\wedge\cdots\wedge T_k\wedge D_{T_i}N,\quad i\in\{1,\ldots,k\}.$$
  The last term of the RHS is zero because $D_{T_i}N = -\lambda T_i$,  and 
  the first term is also zero since $D_{T_i}T_j \perp T_j$.
 Hence $D_{T_i} W = 0$ and  $W$ is constant on $\Sigma^k$ which means that $\Sigma^k$ is in a $(k+1)$-plane $P$.
Finally,  $\Sigma^k$ is in $S^n\cap P$, which is a round $k$-sphere. \end{proof}
Furthermore,  it is shown (cf. for example Prop. 1.19 in \cite{DT}) 
 that the mean curvature vector field of any umbilical submanifold of dimension  at least two is  parallel. \\
 We then deduce from Lemma \ref{7}
 \begin{cor}
 A totally umbilical submanifold   $\Sigma^k \subset \mathbb{R}^{n+1}$ with $k>1$,  is a piece of  a sphere or a  plane.
 \end{cor}
 Note that    a helix with constant curvature and nonzero torsion  is umbilical,
  but the mean curvature vector field is not parallel and the curve is not a circle unless the torsion is  null, i.e., unless 
  the normal bundle is  flat.
  \section{Appendix:  derivatives of symmetric polynomials of curvatures  }\label{symPol}
We  compute some   derivatives of needed symmetric functions of the curvatures using  Proposition \ref{evolution}.
Recall that
$$\Lambda_k = \sum_{i=1}^{n-1} \lambda_i^k,\quad  M_k = 
\sum_{i=1}^{n-1} \mu_i^k,\quad  N_{kl}  = \sum_{i=1}^{n-1} \lambda_i^k\mu_i^l.$$
Hence, we have
\begin{equation*}
\begin{array}{ll}
\Lambda_{1,s}  &= \sum \lambda_{i,s} =  \sum \mu_{i} (\lambda_i-\lambda_n)= N_{11}+M_1\frac{\Lambda_1}{3}; \\
\Lambda_{1,ss} &= N_{11,s}+M_1\frac{\Lambda_{1,s}}{3}+ M_{1,s}\frac{\Lambda_{1}}{3}
=\sum \lambda_{i,s}\mu_i +  \sum \lambda_{i}\mu_{i,s}+\frac{M_1}{3}\left(N_{11}+\frac{M_1\Lambda_1}{3}  \right)  
 +\left(M_2 -\frac{\Lambda_1^2}{3}   \right)\frac{\Lambda_1}{3}\\
\phantom{xxx} &=\sum \mu_i^2(\lambda_i-\lambda_n) +  \sum \lambda_{i}(\mu_i^2+\lambda_i\lambda_n)
+\frac{M_1}{3}\left(N_{11}+\frac{M_1\Lambda_1}{3}  \right)  +\left(M_2 -\frac{\Lambda_1^2}{3}   \right)\frac{\Lambda_1}{3} \\
\phantom{xxx} &=2N_{12} +\frac{M_2\Lambda_1}{3} -\frac{\Lambda_2\Lambda_1}{3}
  +\frac{M_1}{3}\left(N_{11}+\frac{M_1\Lambda_1}{3}  \right)   +\left(M_2 -\frac{\Lambda_1^2}{3}   \right)\frac{\Lambda_1}{3} \\
\phantom{xxx} &= \frac{\Lambda_1}{3} \left(- \Lambda_2+2M_2 +\frac{M_1^2}{3} \right)
 + 2N_{12}  +\frac{M_1}{3}N_{11}   -\frac{\Lambda_1^3}{9};  \\
 \Lambda_{2,s} &=2\sum\lambda_i\mu_i(\lambda_i-\lambda_n)= 2N_{21} +2\frac{\Lambda_1N_{11}}{3} ;\\
 M_{1,s} &= \sum \mu_i^2 +\lambda_i\lambda_n = M_2 -\frac{\Lambda_1^2}{3};\\
 M_{2,s} &= 2 \sum \mu_i(\mu_i^2+\lambda_i\lambda_n)  =2M_3-\frac{2}{3}\Lambda_1N_{11};\\
N_{11,s} &=\sum \mu_i^2(\lambda_i-\lambda_n)+\lambda_i(\mu_i^2+\lambda_i\lambda_n)
= 2N_{12}+\frac{\Lambda_1(M_2-\Lambda_2)}{3};\\
N_{12,s} &=\sum \mu_i^3(\lambda_i-\lambda_n)+ 2\mu_i\lambda_i(\mu_i^2+\lambda_i\lambda_n)
= N_{13}+ M_3\frac{\Lambda_1}{3}+2N_{13}-2N_{21}\frac{\Lambda_1}{3} \\
\phantom{xxx} &= 3N_{13}+ (M_3- 2N_{21})\frac{\Lambda_1}{3}.
\end{array}
\end{equation*}
\textbf{Acknowledgements:} The authors would like to thank Cezar Oniciuc 
for his  guidance, his fruitful comments and suggestions.

 \end{document}